\documentclass[12pt]{article}
\usepackage[utf8]{inputenc}
\usepackage{tikz}
\usepackage{amsthm}
\usepackage{amsmath}
\usepackage{amssymb}
\usepackage{amsfonts}
\usepackage{amsthm}
\usepackage{mathabx}
\usepackage{amscd}
\usepackage[all]{xy}
\usepackage[margin=1in]{geometry}
\usepackage{comment}
\usepackage{pdfpages}
\usepackage{tikz}
\usepackage[T1]{fontenc}
\usepackage{amsmath,amsfonts,amssymb,enumerate}
\usepackage{amscd}
\usepackage{graphicx}
\usepackage[T1]{fontenc}
\usepackage{amsmath,amsfonts,amssymb,enumerate}
\usepackage{amsthm} 
\usepackage[english]{babel}

\usepackage{hyperref}
\hypersetup{
    colorlinks=true,
    linkcolor=blue,
    filecolor=magenta,
    urlcolor=cyan,
    citecolor=green,}
\usepackage[nameinlink,capitalize]{cleveref}

\newtheorem{thm}{Theorem}[section]
\newtheorem{lem}[thm]{Lemma}
\newtheorem{prop}[thm]{Proposition}
\newtheorem{cor}[thm]{Corollary}
\newtheorem{conj}[thm]{Conjecture}

\theoremstyle{definition}

\newtheorem{definition}[thm]{Definition}
\newtheorem{setting}[thm]{Setting}
\newtheorem{remark}[thm]{Remark}
\newtheorem{example}[thm]{Example}
\newtheorem{question}[thm]{Question}
\numberwithin{equation}{section}

\def\min{\mbox{\rm min}}

\def\max{\mbox{\rm max}}
\def\deg{\mbox{\rm deg}}
\def\het{\mbox{\rm ht}}
\def\Spec{\mbox{\rm Spec}}

\title{Rees algebras of ideals of star configurations}

\author{A. Costantini \thanks{University of California, Riverside, Riverside CA 92521, USA.  \emph{e-mail}: \texttt{alessanc@ucr.edu}}, B. Drabkin \thanks{Singapore University of Technology and Design, Singapore, Singapore \emph{e-mail}: \texttt{benjamin\_drabkin@sutd.edu.sg}}, L. Guerrieri \thanks{Jagiellonian University, Instytut Matematyki, 30-348 Krak\'{o}w, Poland. \emph{e-mail}: \texttt{lorenzo.guerrieri@uj.edu.pl}
}
}

\date{$\,$}

\begin{document}

\maketitle

\begin{abstract}
   In this article we study the defining ideal of Rees algebras of ideals of star configurations. We characterize when these ideals are of linear type and provide sufficient conditions for them to be of fiber type. In the case of star configurations of height two, we give a full description of the defining ideal of the Rees algebra, by explicitly identifying a minimal generating set.
\end{abstract}

\section{Introduction}
 
 Ideals of star configurations arise in algebraic geometry, in connection with the study of intersections of subvarieties or subschemes in a projective space. Given a family of hypersurfaces meeting properly in $\mathbb{P}^n$,  a \emph{star configuration of codimension $c$} is the union of all the codimension $c$ complete intersection subschemes obtained by intersecting $c$ of the hypersurfaces (see \cite{GHMN}). The terminology derives from the special case of ten points located at pairwise intersections of five lines in $\mathbb{P}^2$, with the lines positioned in the shape of a star. 
  
 From a commutative algebra perspective, ideals defining star configurations represent an interesting class, since a great amount of information is known about their free resolutions, Hilbert functions and symbolic powers (see for instance \cite{{GHM},{GHMN},{Galetto},{LinShenI},{Paolo},{shifted},{BisuiGHN},{TXie},{LinShenII}}). In this article we study their Rees algebras, about which little is currently known (see for instance \cite{{HHVladoiu},{GSimisT},{Nicklasson19},{BTXie}}).

 If $I = (g_1, \ldots, g_{\mu})$ is an ideal in a Noetherian ring $R$, the \emph{Rees algebra} of $I$ is the subalgebra 
    $$ \mathcal{R}(I) \coloneq R[It] = R[g_1 t, \dots, g_{\mu}t] \subseteq R[t]$$
 of the polynomial ring $R[t]$. In particular, $\,g_1 t, \dots, g_{\mu}t\,$ are $R$-algebra generators of $\mathcal{R}(I)$, and the algebraic structure of $\mathcal{R}(I)$ is classically understood by determining the ideal of relations among these generators. The latter is called the \emph{defining ideal} of the Rees algebra, and its generators are called the \emph{defining equations} of $\mathcal{R}(I)$. Geometrically, $\mathrm{Proj}(\mathcal{R}(I))$ is the blow-up of the affine scheme $X=\Spec(R)$ along the subscheme $V(I)$. 
  
 Determining the defining ideal of Rees algebras is usually difficult. Indeed, although the defining equations of degree one can be easily determined from a presentation matrix of the given ideal, a full understanding of the defining ideal of $\mathcal{R}(I)$ often requires prior knowledge of a free resolution of $I$ and of its powers $I^m$. 
 On the other hand, only a few classes of ideals have well-understood free resolutions, and usually a free resolution for $I$ does not provide information on the free resolutions of its powers $I^m$. Nevertheless, the problem becomes manageable if one imposes algebraic conditions on a presentation matrix of $I$ (see for instance \cite{{UVeqLinPres},{MU},{KPU}}), especially when one can exploit methods from algebraic combinatorics (see for instance \cite{{Villarreal},{HHVladoiu},{FouliLin},{twoBorel},{GSimisT},{HaMorey},{ALLin}}).

 The rich combinatorial structure of ideals of star configurations sometimes allows to deduce information on their Rees algebra. For instance, \emph{monomial star configurations}, which are constructed choosing the hypersurfaces to be coordinate hyperplanes in $\mathbb{P}^n$, are monomial ideals associated with discrete polymatroids, and hence are of fiber type by work of Herzog, Hibi and Vladoiu (see \cite[3.3]{HHVladoiu}). Recall that an ideal $\,I \subseteq R= K[x_1, \ldots, x_n]\,$ is said to be of \emph{fiber type} if the non-linear equations of the Rees algebra $\mathcal{R}(I)$ are given by the defining equations of the \emph{fiber cone} $\,F(I) \coloneq \mathcal{R}(I) \otimes_R K$. 
 
 In order to study \emph{linear star configurations}, which are constructed choosing arbitrary hyperplanes in $\mathbb{P}^{n}$, one can instead exploit the combinatorial properties of hyperplane arrangements. In particular, Garrousian, Simis and Toh\u{a}neanu proved that ideals of this kind of height two are of fiber type, and provided a (non-minimal) generating set for the defining ideal of their Rees algebra (see \cite[4.2 and 3.5]{GSimisT}). Similar results were obtained for linear star configurations of height three in $\mathbb{P}^2$ by Burity, Toh\u{a}neanu and Xie (see \cite[3.4 and 3.5]{BTXie}), who also conjectured that ideals of linear star configurations of any height are of fiber type.

 In this context, it is then natural to ask the following question. 
  
 \begin{question}
   \label{QuestionFiberType} \hypertarget{QuestionFiberType}{}
      Let $\mathcal{F} = \lbrace F_1, \ldots, F_t \rbrace$ be a family of homogeneous polynomials in $K[x_1, \ldots, x_n]$ and let $I_{c,\mathcal{F}}$ be the ideal of the star configuration of height $c$ obtained from the hypersurfaces defined by the $F_i$'s. Under what conditions on $\mathcal{F}$ is $I_{c,\mathcal{F}}$ of fiber type? 
 \end{question}
 
 Although in general ideals of star configurations may not be of fiber type, we show that this is always the case when the elements of $\mathcal{F}$ form a regular sequence. More precisely, our first main result (see \cref{powerstarconf} and \cref{degree2}) is the following.
 
 \begin{thm}
  \label{mainsec3} \hypertarget{mainsec3}{}
    Let $\mathcal{F} = \lbrace F_1, \ldots, F_t \rbrace$ be a homogeneous regular sequence in $K[x_1, \ldots, x_n]$. Let $I$ be the ideal of a star configuration of height $c \geq 2$ constructed on the hypersurfaces defined by the elements of $\mathcal{F}$. Then for any $m \geq 1$, $I^m$ is of fiber type. Moreover, the defining equations of the fiber cone $F(I^m)$ have degree at most two.  
 \end{thm}
 
 Our key observation is that, under these assumptions, $I$ and its powers $I^m$ are generated by \emph{monomials in the $F_i$'s}, i.e. elements of the form $\, F_1^{i_1} \cdots F_t^{i_t}$. Hence, the defining ideal of the Rees algebra of $I^m$ can be deduced from its Taylor resolution (see \cite[Chapter IV]{Taylor}). This method was previously used by several authors to study Rees algebras of squarefree monomial ideals (see for instance \cite{{Villarreal},{FouliLin},{Kumar2},{HaMorey}}). 
 We remark that the content of \cref{mainsec3} was already known in the case when $\mathcal{F} = \lbrace x_1, \ldots, x_n \rbrace$ (see \cite[3.3]{HHVladoiu} and \cite[5.3(b)]{HHpolymatroids}). However, our proof is substantially different, since we only perform algebraic manipulations on the generators of $I^m$, while the proof of \cite[5.3(b)]{HHpolymatroids} heavily relies on the use of Gr\"obner bases and monomial orders (via results of De Negri \cite[2.5 and 2.6]{DeNegri}, see \cite[5.2 and 5.3]{HHpolymatroids}).

 In the light of \cref{mainsec3}, it is natural to explore how moving away from the assumption that the elements of $\,\mathcal{F} = \lbrace F_1, \ldots, F_t \rbrace\,$ form a regular sequence affects the structure of the defining ideal of the Rees algebra $\mathcal{R}(I_{c, \mathcal{F}})$. Our second main result (see \cref{Gs} and \cref{nonlineartypelocus}) shows that the length of a maximal regular sequence in the family $\mathcal{F}$ characterizes the linear-type property of ideals of linear star configurations. Recall that an ideal is said to be of \emph{linear type} if the defining ideal of its Rees algebra only consists of linear equations. 
 
 More deeply, when the $F_i$'s are all linear, in \cref{sregseq-minorsU} we characterize the regular sequences contained in $\mathcal{F}$ in terms of the matrix whose entries are the coefficients of the $F_i$'s. This turns out to be particularly useful in the case of linear star configurations of height two. Indeed, an ideal $I$ of this kind is perfect, hence it is classically known that the maximal minors of the \emph{Jacobian dual} of a Hilbert-Burch matrix of $I$ are defining equations of the Rees algebra $\mathcal{R}(I)$ (see \cref{sectionbackground} for the definition of Jacobian dual matrices). Thanks to \cref{sregseq-minorsU}, we can interpret the defining ideal described by Garrousian, Simis and Toh\u{a}neanu in \cite[4.2 and 3.5]{GSimisT} in terms of the associated primes of the ideal of maximal minors of the Jacobian dual. More precisely, our main result (see \cref{orlikterao} and \cref{intersection}) is the following.
 
 \begin{thm}
   \label{mainsec6} \hypertarget{mainsec6}{}
     Let $I$ be the ideal of a linear star configuration of height two. Then, the defining ideal of the Rees algebra $\mathcal{R}(I)$ is $\mathcal{L}+ \mathcal{P}$, where $\mathcal{L}$ consists of linear equations and (under mild assumptions) $\mathcal{P}$ is the only associated prime of the ideal of maximal minors of a Jacobian dual for $I$ that is not generated by monomials. 
  \end{thm}

 The proof of \cref{mainsec6} proceeds in three steps. First, we identify a minimal generating set for the ideal of maximal minors of a Jacobian dual for $I$ (see \cref{minorsJacDual}). Next, we prove that suitable irreducible factors of these minimal generators span all the non-linear equations described in \cite{GSimisT} (see \cref{orlikterao}), hence they are the minimal non-linear equations of the Rees algebra. Finally, using \cref{sregseq-minorsU} and under mild assumptions on the matrix of coefficients of the $F_i$'s, we provide a primary decomposition of the ideal of maximal minors of the Jacobian dual and show that $\mathcal{P}$ satisfies the required property (see \cref{intersection}). Our approach is entirely algebraic, combining linear algebra with divisibility arguments. As a biproduct, it allows us to identify the degrees of each non-linear generator of the defining ideal of $\mathcal{R}(I)$ (see \cref{degreesOT}), which was not obvious from the combinatorial arguments appearing in the proof of \cite[3.5]{GSimisT}.
 
 We remark that our methods can be potentially extended to ideals of linear star configurations of height greater than two, since the notion of Jacobian dual matrix is defined with no restriction on the height of the ideal under examination (see \cite[p. 191]{VasconcelosJac}). Moreover, it seems reasonable to believe that the defining ideal of the fiber cone can be determined from the associated primes of the ideal of maximal minors of the Jacobian dual matrix in other cases as well. In fact, the defining ideal of the fiber cone is sometimes known to coincide with the radical of the ideal of maximal minors of a Jacobian dual. This is, for instance, the case for certain equigenerated ideals with a linear presentation (including perfect ideals of height three that are linearly presented, see \cite[7.1, 7.2 and 7.4]{KPU}), or certain equigenerated homogeneous ideals of arbitrary height (see \cite[3.5]{HHVladoiu}).

 \bigskip 
 
 We now describe how this article is structured. 
 
 In \cref{sectionbackground} we collect the necessary background on ideals of star configurations and Rees algebras that we will use throughout. In
 \cref{sectionregular} we study the Rees algebras of (powers of) ideals of star configurations defined with respect to a regular sequence $\,\mathcal{F}= \lbrace  F_1, \ldots, F_t \rbrace$. The main results of this section are \cref{powerstarconf} and \cref{degree2}.
 
 \cref{sectionLinearType} is devoted to the linear type property of ideals of star configurations. In particular, \cref{nonlineartypelocus} provides a criterion for such an ideal to be of linear type. In the case of linear star configurations, we also fully characterize when these ideals are of linear type locally up to a certain height (see \cref{Gs} and \cref{nonlineartypelocus}). 
 
 In the remaining sections we focus on linear star configurations of height two. In \cref{sectionJacobiandual} we construct a minimal generating set for the ideal of maximal minors of their Jacobian dual, which we exploit in \cref{sectionheight2} in order to characterize the defining ideal of the Rees algebra of linear star configurations of height two (see \cref{orlikterao} and \cref{intersection}).

\section{Background}
 \label{sectionbackground} \hypertarget{sectionbackground}{}

 In this section we collect the necessary background information on ideals of star configurations and Rees algebras. 
 
\subsection{Star configurations}
  
 Throughout this article $R=K[x_1, \ldots, x_n]\,$ denotes a standard graded polynomial ring over a field $K$ and $\,\mathcal{F}= \lbrace  F_1, \ldots, F_t \rbrace\,$ denotes a set of homogeneous elements in $R$. For any integer $c\,$ with $\,1 \leq c \leq t$, let
   $$ I_{c,\mathcal{F}}= \bigcap_{1 \leq i_1 < \ldots < i_c \leq t} (F_{i_1}, \ldots, F_{i_c}), $$
   
 \begin{definition}
   If $\,1 \leq c \leq \mathrm{min}\{n,t\}$ and any subset of $\mathcal{F}$ of cardinality $c+1$ is a regular sequence, then $\, I_{c,\mathcal{F}} \,$ is called the \emph{ideal of the star configuration of height $c$ on the set $\mathcal{F}$}. 
    
   We say that $\, I_{c,\mathcal{F}} \,$ defines a \emph{linear star configuration} if in addition all the $F_i$ are homogeneous of degree one, and a \emph{monomial star configuration} if $\, \mathcal{F} = \lbrace x_1, \ldots, x_n \rbrace$.
 \end{definition}
 
 The following proposition summarizes useful results about ideals of star configurations.  
 
 \begin{prop}
   \label{generatorsI} \hypertarget{generatorsI}{}
      Let $I_{c,\mathcal{F}}$ be the ideal of the star configuration of height $c$ on the set $\,\mathcal{F} = \lbrace  F_1, \ldots, F_t \rbrace$. For $s \leq t$, denote $\, \displaystyle{J_{s,\mathcal{F}} = \sum_{1 \leq i_1 < \ldots < i_s \leq t} (F_{i_1} \cdots F_{i_s})}$. Then:
     \begin{enumerate}[\rm(1)]
         \item {\rm(\cite[2.3]{GHMN})} $\, \displaystyle{\, I_{c,\mathcal{F}} =  J_{t-c+1,\mathcal{F}} }$. In particular, the minimal number of generators of $I_{c,\mathcal{F}}$ is $\,\mu(I_{c,\mathcal{F}}) = \binom{t}{t-c+1} \geq t,\,$ and $\,\mu(I_{c,\mathcal{F}}) = t\,$ if $\,c=2$.
         \item {\rm (\cite[3.3 and 3.5]{GHMN})} $\,I_{c,\mathcal{F}}\,$ is a perfect ideal and has a linear resolution.
        \item {\rm (\cite[2.2 and 3.1]{BTXie} and \cite[3.1, 5.2 and 5.3]{ConcaH})} $\,$If the $F_i$'s are all linear or $\, \mathcal{F}= \lbrace x_1, \ldots, x_n \rbrace$, then $\,I_{c,\mathcal{F}}\,$ has linear powers (i.e.$\!$ for every $m$, $\,I^m_{c,\mathcal{F}}\,$ has a linear resolution). 
     \end{enumerate}
 \end{prop}
 
   

\subsection{Rees algebras}
 
 Although some of the definitions included in this subsection make sense over any Noetherian ring and for any ideal, we assume throughout that $R \coloneq K[x_1, \ldots, x_n]$ is a polynomial ring over a field $K$ and that $\,I=(g_1, \ldots, g_{\mu})\,$ is a graded ideal of $R$. The \emph{Rees algebra} of $I$ is the subalgebra 
   $$ \mathcal{R}(I) \coloneq \bigoplus_{i \geq 0} I^i t^i \subseteq R[t]$$
 of the polynomial ring $R[t]$, where $I^0 \coloneq R$. Notice that $\, \mathcal{R}(I) = R[g_1 t, \dots, g_{\mu}t] \,$ is a graded ring (with grading inherited from that of $R[t]$) and there is a natural graded epimorphism 
     $$  \varphi \colon S \coloneq R[T_1, \ldots, T_{\mu}] \twoheadrightarrow \mathcal{R}(I) ,$$
 defined by $\,\varphi(T_i) = g_i t\,$ for all $\,1 \leq i \leq {\mu}$. The ideal $\mathcal{J} \coloneq \mathrm{ker}(\varphi)$ is called the \emph{defining ideal} of the Rees algebra $\mathcal{R}(I)$, and the generators of $\mathcal{J}$ are called the \emph{defining equations} of $\mathcal{R}(I)$.  Notice that $\,\displaystyle{\mathcal{J} = \bigoplus_{s \geq 0} \mathcal{J}_s}\,$ is a graded ideal (in the $T_i$ variables) of the polynomial ring $S$.
 
 The linear equations of $\mathcal{R}(I)$ can be easily determined from a presentation matrix of $I$. Specifically, if 
    $$ R^s \stackrel{M}{\longrightarrow} R^{\mu} \longrightarrow I \to 0 $$
 is a presentation of $I$, then $\mathcal{J}_1 = (\lambda_1, \ldots, \lambda_s)$, where the $\lambda_i$'s are homogeneous linear polynomials in $S$ satisfying the matrix equation
   $$ [\lambda_1, \ldots, \lambda_s] = [T_1, \ldots, T_{\mu}] \cdot M. $$
 We denote $\mathcal{J}_1$ by $\mathcal{L}$. The ideal $I$ is said to be \emph{of linear type} if $\mathcal{J} = \mathcal{L}$. 
 
 Given an arbitrary ideal $I$, one should expect that the defining ideal of the Rees algebra $\mathcal{R}(I)$ also contains non-linear equations. The latter are usually difficult to determine, however if the $g_i$ all have the same degree, the non-linear equations can sometimes be identified by analyzing the \emph{fiber cone} of $I$, which is defined as
   $$ F(I) \coloneq \mathcal{R}(I) \otimes_R K = K[g_1, \dots, g_{\mu}] \cong \frac{K[T_1, \ldots, T_{\mu}]}{\mathcal{I}} .$$
 \noindent Indeed, by construction one always has $\,\mathcal{L} + \mathcal{I}S \subseteq \mathcal{J},\,$ and the ideal $I$ is called of \emph{fiber type} when the latter inclusion is an equality. In fact, the rings $S$ and $\mathcal{R}(I)$ can be given natural structures of bigraded $K$-algebras by setting $\deg (T_i) = (0,1)$ and $\deg (x_i) = (d_1, 0)$, where $d_i$ is the degree of $x_i$ in $R$. Then, $\,\displaystyle{\mathcal{J} = \bigoplus_{i,j \geq 0} \mathcal{J}_{(i,j)}}\,$ is a bigraded ideal and $I$ is of fiber type if and only if the defining ideal of $\mathcal{R}(I)$ is 
   $$ \mathcal{J} =  [\mathcal{J}]_{(\ast,1)} + [\mathcal{J}]_{(0,\ast)}.$$

 When $I$ is a perfect ideal of height two, one can best exploit the information contained in a Hilbert-Burch presentation matrix $M$ of $I$ through the notion of a Jacobian dual matrix, which was first introduced in \cite{VasconcelosJac}. Recall that the \emph{Jacobian dual} $B(M)$ of $M$ is an $ n \times (\mu -1)$ matrix with coefficients in $S$, satisfying the equation
   $$ [x_1, \ldots, x_n] \cdot B(M) = [T_1, \ldots, T_{\mu}] \cdot M. $$
 
 A priori one could find several possible Jacobian duals associated with a given presentation $M$, however $B(M)$ is uniquely determined in the case when the entries of $M$ are linear polynomials in $R=K[x_1, \ldots, x_n]$.  Moreover, if $\,s = \max \{ n, \mu -1 \} \,$ and $I_s (B(M))$ denotes the ideal of maximal minors of $B(M)$, then the defining ideal $\mathcal{J}$ of $\mathcal{R}(I)$ always satisfies the inclusion 
  \begin{equation}
    \label{eqJacdual} \hypertarget{eqJacdual}{}
      \mathcal{L} + I_s (B(M)) \subseteq \mathcal{J}
  \end{equation}

 Notice that, if $M$ has linear entries, then the entries of $B(M)$ are in $K[T_1, \ldots, T_{\mu}]$. Hence, the images of the generators of $I_s (B(M))$ in the fiber cone $F(I)$ are in the defining ideal $\mathcal{I}$ of $F(I)$. Although the inclusion in \cref{eqJacdual} is usually strict, \cref{morey-ulrich} below provides an instance when equality holds. Before we state the theorem we need to recall the following definition, which we will use often throughout this article. 
 
 \begin{definition}
   \label{defGs} \hypertarget{defGs}{}
   An ideal $I$ satisfies the \emph{$G_s$ condition} if $\mu(I_{\mathfrak{p}}) \leq \dim R_{\mathfrak{p}}$ for every $\mathfrak{p} \in V(I)$ of height at most $s-1$. Equivalently, if $M$ is any presentation matrix of $I$, then $I$ satisfies the $G_s$ condition if and only if for every $\,1 \leq i \leq s-1$, $\,\het (I_{\mu(I)-i}(M)) \geq i+1$.
   
   \noindent Moreover, $I$ is said to satisfy $G_{\infty}$ if it satisfies the $G_s$ condition for every integer $s$.
 \end{definition}

\begin{remark}
  \label{lintypeGinfty} \hypertarget{lintypeGinfty}{}
  Ideals of linear type always satisfy $G_{\infty}$.
\end{remark}
 Under the assumption that $I$ satisfies $G_n$, the Rees algebra of $I$ is described by the following result of Morey and Ulrich. Recall that the Krull dimension of the fiber cone $F(I)$ is called the \emph{analytic spread} of $I$ and is denoted with $\ell(I)$.
 \begin{thm}[{\rm \cite[1.3]{MU}}] 
  \label{morey-ulrich} \hypertarget{morey-ulrich}{}
   Let $R=K[x_1, \ldots, x_n]$ and assume that $K$ is infinite. Let $I \subseteq R$ be a linearly presented perfect ideal of height 2 with $\mu(I) \geq n+1$ and assume that $I$ satisfies $G_n$. Let $M$ be a Hilbert-Burch matrix resolving $I$, then the defining ideal of $\mathcal{R}(I)$ is 
     $$\mathcal{J}= \mathcal{L}+I_n(B(M))$$ 
   where $B(M)$ is the Jacobian dual of $M$. Moreover, $\mathcal{R}(I)$ and $F(I)$ are Cohen-Macaulay, $I$ is of fiber type and $\ell(I)=n$.  
 \end{thm}

 In \cref{sectionheight2} we will also use some results on Rees algebras of ideals generated by $a$-fold products of linear forms from \cite{GSimisT}, which we briefly recall here. Let $\,\mathcal{F} = \lbrace  F_1, \ldots, F_t \rbrace$ be a set of homogeneous linear polynomials in $R=K[x_1, \ldots, x_n]$ and suppose that $t \geq n$. 
 Notice that each $F_i$ defines a line through the origin in $\mathbb{P}^{n-1}$, hence the set $\mathcal{F}$ defines a \emph{central hyperplane arrangement} in $\mathbb{P}^{n-1}$. A \emph{linear dependency} among $s$ of the given linear forms is a relation of the form
 \begin{equation}
   \label{dependency} \hypertarget{dependency}{}
     D \colon c_{i_1} F_{i_1} + \ldots + c_{i_s} F_{i_s} = 0.
 \end{equation}
 Given a linear dependency $D$, one can define the following homogeneous polynomial in $S=R[T_1, \ldots, T_{\mu}]$
 \begin{equation}
    \label{deltadependency} \hypertarget{deltadependency}{}
     \partial D \colon \sum_{j=1}^s c_{i_j} \prod_{j \neq k=1}^s T_{i_k} .
 \end{equation}
 
 The following result of Garrousian, Simis and Toh\u{a}neanu relates the $\partial D$'s to the fiber cone and Rees algebra of ideals generated by $(t-1)$-fold products of linear forms.
 \begin{thm} 
    \label{garrousian-simis-tohaneanu} \hypertarget{garrousian-simis-tohaneanu}{}
   With the notation above, let $\, \displaystyle{ I = \sum_{1 \leq i_1 < \ldots < i_{t-1} \leq t} (F_{i_1} \cdots F_{i_{t-1}}) }$. Then,
    \begin{enumerate}[\rm(1)]
        \item {\rm (\cite[4.2]{GSimisT})} $I$ is of fiber type.
        \item {\rm (\cite[3.5]{GSimisT})} The defining ideal of the fiber cone $F(I)$ is generated (possibly not minimally) by all elements of the form $\partial D$ as in \cref{deltadependency}, where $D$ varies within the set of linear dependencies among any $\,t-1$ of the $F_i$'s. 
        \item {\rm (\cite[4.9 and 4.10]{GSimisT})} $\mathcal{R}(I)$ and $F(I)$ are Cohen-Macaulay.
    \end{enumerate}
 \end{thm}
 
 

\section{Star configurations on a regular sequence} 
  \label{sectionregular} \hypertarget{sectionregular}{}
  
  In this section we assume the following setting.

 \begin{setting}
  \label{settingregular} \hypertarget{settingregular}{}
    Let $R=K[X_1, \ldots, X_n]\,$ be a polynomial ring over a field $K$ and let $\,\mathcal{F}= \lbrace  F_1, \ldots, F_t \rbrace\,$ be a homogeneous regular sequence in $R$. For $\,1 \leq c \leq \mathrm{min}\{n,t\}$, let
      $$ I= I_{c,\mathcal{F}}= \bigcap_{1 \leq i_1 < \ldots < i_c \leq t} (F_{i_1}, \ldots, F_{i_c}) $$
    be the ideal of the star configuration of height $c$ on the set $\mathcal{F}$. For a fixed $m \geq 1$, let $\,g_1, \ldots, g_{\mu}$ be a minimal generating set for the $m$-th power $I^m$ of $I$, and write $S=R[T_1, \ldots, T_{\mu}]$.
 \end{setting}

 Since the elements of $\mathcal{F}$ form a regular sequence, from \cref{generatorsI} it follows that the $I$ is minimally generated by \emph{monomials in the $F_i$}, i.e. products $\,F_1^{i_1} \ldots F_{n}^{i_n}\,$ for some integers $i_j \geq 0$. Notice that the exponents $i_j$ are uniquely determined because the $F_i$'s form a regular sequence. Moreover, since every $m$-th power $I^m$ is minimally generated by $m$-fold products of minimal generators of $I$, with the assumptions and notations of \cref{settingregular} we may also assume that $\,g_1, \ldots, g_{\mu\,}$ are monomials in the $F_i$'s. 
 
 The defining ideal of $\, \mathcal{R}(I^m)$ can then be determined from the Taylor resolution of $I^m$ (see \cite[Chapter IV]{Taylor}). More precisely, for $\,1 \leq s \leq t\,$ let 
   $$\mathcal{I}_{s} = \lbrace \alpha =(i_1, \ldots, i_s) \, | \, 1 \leq i_1 \leq \ldots \leq i_s \leq t \rbrace .$$
 For each $\,\alpha \in \mathcal{I}_{s}\,$ denote  $\,g_{\alpha}= g_{i_1} \cdots g_{i_s},\,$ $\,T_{\alpha}= T_{i_1} \cdots T_{i_s}\,$ and 
   \begin{equation}
     \label{Talphabeta} \hypertarget{Talphabeta}{}
       T_{\alpha, \beta} = \frac{g_{\beta}}{\gcd(g_{\alpha}, g_{\beta})} \,T_{\alpha}- \frac{g_{\alpha}}{\gcd(g_{\alpha}, g_{\beta})}\, T_{\beta}.
   \end{equation}
 Then,  the defining ideal of $\,\mathcal{R}(I^{m\,})$ is
   $$ \mathcal{J} = \mathcal{J}_1 + \bigcup_{s=2}^{\infty} \mathcal{J}_s,  $$
 where $\, \displaystyle{\mathcal{J}_s = \lbrace T_{\alpha, \beta} \, | \, \alpha, \beta \in \mathcal{I}_{s} \,  \rbrace} \,$ for every $s \geq 2\,$.
   

 Since $\mathcal{R}(I^{m\,})$ is Noetherian, there exists an $N \geq 2$ so that $\,\displaystyle{\mathcal{J}_s \subseteq \sum_{i=1}^N \mathcal{J}_i}\,$ for all $s\geq N$. In order to estimate such an $N$, the first step is to identify redundant relations, which we do in the following lemma. To simplify the notation, for multi-indices $\alpha = (i_1, \cdots , i_s)$ and  $\beta= (j_1,\cdots, j_s)$, denote $\,\theta= \displaystyle{ \frac{g_{\beta}}{\gcd(g_{\alpha}, g_{\beta})}}\,$ and $\,\delta= \displaystyle{ \frac{g_{\alpha}}{\gcd(g_{\beta}, g_{\alpha})}}$. Then, \cref{Talphabeta} can be rewritten as
  \begin{equation}
    \label{relation} \hypertarget{relation}{}
    T_{\alpha,\beta} =  \theta\, T_{i_1}\cdots T_{i_s} - \delta\, T_{j_1}\cdots T_{j_s} 
  \end{equation}

 \begin{lem}
  \label{assumption1} \hypertarget{assumption1}{}
    With the notation above, let $\,\displaystyle{\theta_1 \coloneq \frac{g_{j_1}}{\gcd(g_{i_1},g_{j_1})}}$. Assume that, up to reordering the indices $i_k$ and $j_k$ in \cref{relation}, $\theta_1$  divides $\theta$. Then, for $s \geq 2$ the relation in \cref{relation} can be expressed as an $S$-linear combination of relations in $\mathcal{J}$ of degree at most $s-1$.  
 \end{lem}

 \begin{proof}
   Write $\theta = \theta_1 \theta'$ and define $\,\displaystyle{\delta_1 \coloneq \frac{g_{i_1}}{\gcd(g_{i_1},g_{j_1})}}$. Notice that $\theta_1 T_{i_1} - \delta_1T_{j_1} \in \mathcal{J}_1$. Using this linear relation we can rewrite
     $$ \theta \,T_{i_1}\cdots T_{i_s} - \delta \,T_{j_1}\cdots T_{j_s} = \theta'\,T_{i_2}\cdots T_{i_s}(\theta_1 T_{i_1} - \delta_1T_{j_1}) - T_{j_1}( \delta \,T_{j_2}\cdots T_{j_s} -\delta_1 \theta'\,T_{i_2}\cdots T_{i_s}).$$
   Now, the term $\,\theta_1 T_{i_1} - \delta_1T_{j_1}\,$ is clearly in $\mathcal{J}_1$ and the term $\,\delta \,T_{j_2}\cdots T_{j_s} -\delta_1 \theta'\,T_{i_2}\cdots T_{i_s}\,$ is in $\mathcal{J}_{s-1}$. Indeed, observe that
     $$ \theta' g_{i_2}\cdots g_{i_s} = \frac{\theta g_{\alpha} }{\theta_1g_{i_1}} =  \frac{\delta g_{\beta}}{\delta_1g_{j_1}} = \frac{\delta g_{j_2}\cdots g_{j_s}}{\delta_1}, $$ 
   and therefore $\, \delta_1 \theta' g_{i_2}\cdots g_{i_s} = \delta g_{j_2}\cdots g_{j_s}. $ 
 \end{proof}
 
 Notice that, whenever there exists a grading so that the $F_i \in \mathcal{F}$ have the same degree, then $I$ is generated by elements of the same degree $\alpha(I)$.
 Then, each power $I^m$ is generated by all the monomials in the $F_i$'s of degree $\,\alpha(I)m\,$ that are not divisible by $F^{m+1}$ for any $F \in \mathcal{F}$. Moreover, the fiber cone $F(I^m)$ is isomorphic to the \emph{toric ring} of $\,\lbrace g_1, \ldots, g_{\mu} \rbrace$. In particular, its defining ideal is generated by all binomials of the form $\, T_{\alpha} - T_{\beta}\, $ so that $\, g_{\alpha} = g_{\beta}$ (see \cite[Proposition 10.1.1]{HH} for a proof).

 \begin{thm}
  \label{powerstarconf} \hypertarget{powerstarconf}{}
   With the assumptions and notations of \cref{settingregular}, assume also that $I$ is generated by elements of the same degree with respect to some $\mathbb{Z}_{> 0}$-grading. Then, the ideal $I^m$ is of fiber type.
 \end{thm}

 \begin{proof}
 We need to prove that all the relations of the form described in \cref{relation} can be expressed as an $S$-linear combination of linear relations and fiber-type relations. Working by induction on $s$, it suffices to show that any relation $\,\displaystyle{T_{\alpha,\beta} =  \theta\, T_{i_1}\cdots T_{i_s} - \delta \,T_{j_1}\cdots T_{j_s}}\,$ of degree $s \geq 2$ that is not of fiber type can be expressed as an $S$-linear combination of relations of smaller degree in the $T_i$ variables. 
   
 Since $T_{\alpha,\beta}$ is not of fiber type, there exists a form $F_1 \in \mathcal{F}$ so that $F_1$ divides $\theta$. Moreover, since for any $F \in \mathcal{F}$ we know that $F^{m+1}$ cannot divide any generator of $I^m$, by possibly relabeling the indices, we may assume that $F_1$ divides $g_{j_1}$ and $F_1^m$ does not divide $g_{i_1}$.
 As in \cref{assumption1}, let $\,\displaystyle{ \theta_1 \coloneq \frac{g_{j_1}}{\gcd(g_{i_1},g_{j_1})}}$. If $\theta_1$ divides $\theta$, we are done by \cref{assumption1}. Hence we assume the opposite condition. Write 
     $$\theta_1 = F_1^{\,p_1} \cdots F_a^{\,p_a} G_1^{\,q_1} \cdots G_b^{\,q_b} $$
 where the $F_k$'s and $G_k$'s are elements of $\mathcal{F}$ such that, for all $k$, $F_k^{\,p_k}$ divides $\theta$, while no $G_k^{\,q_k}$ divides $\theta$. Notice that $p_1 \geq 1$ by what was said above about $F_1$. Moreover, since $T_{\alpha,\beta}$ does not satisfy the assumption of \cref{assumption1}, necessarily we must have that $b \geq 1$, so there exists at least one such form $G_1$. In particular, notice that $G_1$ divides $g_{j_1}$ but $G_1^{\,m}$ does not divide $g_{i_1}$.
 
 Similarly, let $\,\displaystyle{\delta_1 \coloneq  \frac{g_{i_1}}{\gcd(g_{i_1},g_{j_1})}}\,$. If $\delta_1$ divides $\delta$, the thesis follows by applying \cref{assumption1} to $\delta$ and $\delta_1$. So, assume that $\delta_1$ does not divide $\delta$. Then, there exist a form $H_1 \in \mathcal{F}$ and an integer $r_1 \geq 1$ such that $H_1^{\,r_1}$ divides $\delta_1$ but does not divide $\delta$. In particular, $H_1$ divides $g_{i_1}$ and $H_1^m$ does not divide $g_{j_1}$. Moreover, since $\gcd(\theta_1,\delta_1)=1$, we also know that $H_1$ does not divide $\theta_1$ and $G_1$ does not divide $\delta_1$. 
 
 Now, rewriting $ \, \displaystyle{g_{i_1} =  \gcd(g_{i_1}, g_{j_1}) \delta_1}$  and $ \, \displaystyle{g_{j_1} =  \gcd(g_{i_1}, g_{j_1}) \theta_1}$, we get that
 \begin{equation}
   \label{gik-gkkequation} \hypertarget{gik-gkkequation}{}
     \theta\, \delta_{1\,} g_{i_2}\cdots g_{i_s} = \delta \, \theta_{1\,} g_{j_2}\cdots g_{j_s}
 \end{equation}
 We claim that either there exists a $g_{j_k}$ with $k \geq 2$ such that $H_1$ divides $g_{j_k}$ and $G_1^m$ does not divide $g_{j_k}$, or there exists a $g_{i_k}$ such that $G_1$ divides $g_{i_k}$ and $H_1^m$ does not divide $g_{i_k}$. Indeed, suppose that for all $k\geq 2$, $H_1$ divides $g_{j_k}$ if and only if $G_1^m$ divides $g_{j_k}$ and $H_1^m$ divides $g_{i_k}$ if and only if $G_1$ divides $g_{i_k}$. Assume that $H_1$ divides exactly $c$ of the $g_{j_k}$'s for $k \geq 2$ and that $G_1$ divides exactly $d$ of the $g_{i_k}$'s for $k \geq 2$. Then, the degree of $G_1$ on each side of \cref{gik-gkkequation} is at least $q_1 +cm$ and at most $q_1 -1 +dm$. Similarly, the degree of $H_1$ on each side of \cref{gik-gkkequation} is at least $r_1 +dm$ and at most $r_1 -1 +cm$. Hence, we must simultaneously have that $cm \leq dm -1$ and $dm \leq cm -1$, which is impossible. This proves our claim. 
 Up to relabeling the $j_k$'s or $i_k$'s, we may assume that $k=2$. 

 From the discussion above it follows that there exist generators $g_{h_1},g_{h_2}$ of $I^m$ such that either $H_1g_{j_1} = G_1g_{h_1}$ and $H_1g_{h_2} = G_1g_{j_2}$, or $G_1g_{i_1} = H_1g_{h_1}$ and $G_1g_{h_2} = H_1g_{i_2}$. 
 Let us consider the first case (the second case is equivalent). We can write 
  \begin{eqnarray*}
    T_{\alpha,\beta} & = & \theta T_{i_1}\cdots T_{i_s} - \delta T_{j_1}\cdots T_{j_s}\\
   & = & \theta T_{i_1}\cdots T_{i_s} - \delta T_{h_1}T_{h_2}\cdots T_{j_s} + \delta T_{j_3}\cdots T_{j_s}( T_{h_1}T_{h_2} -  T_{j_1}T_{j_2}).
  \end{eqnarray*}
 The last term is a multiple of a fiber-type relation by a monomial in $S$. Moreover, from our choice of $h_1$ and $h_2$ it follows that $\,\displaystyle{\delta = \frac{g_{\alpha}}{\gcd(g_{\alpha}, g_{\gamma})} },\,$ where $\gamma= (h_1, h_2, j_3, \ldots, j_s)$. Hence, 
    $$ T_{\alpha,\beta}^{(1)}= \theta T_{i_1}\cdots T_{i_s} - \delta T_{h_1}T_{h_2}\cdots T_{j_s} \in \mathcal{J}_s$$ 
 and we only need to prove our claim for this new relation.
Set $\,\displaystyle{\theta_1^{(1)} \coloneq \frac{g_{h_1}}{\gcd(g_{i_1},g_{h_1})}=\frac{\theta_1}{G_1}}$ (where the latter equality holds because $\,\displaystyle{\gcd(g_{i_1}, g_{h_1}) = H_1 \gcd(g_{i_1}, g_{j_1})}$).  
Then there are two possibilities: either this new relation $T_{\alpha,\beta}^{(1)}$ satisfies the assumption of \cref{assumption1} and the proof is complete, or we iterate the process considering another form $G_k\in\{G_1,\ldots,G_b\}$. As before, by subtracting a multiple of a fiber-type relation from $T_{\alpha,\beta}^{(1)}$, we reduce to a new relation $T_{\alpha,\beta}^{(2)}$ such that the term corresponding to $\theta_1^{(1)}$ is $\,\displaystyle{\theta_1^{(2)} \coloneq \frac{\theta_1^{(1)}}{G_k}}$. Since the monomial $F_1^{p_1} \cdots F_a^{p_a}$ divides $\theta$, iterating this argument finitely many times, we reduce to prove our claim for a relation satisfying the assumption of \cref{assumption1}. 
\end{proof}
 

 \begin{thm}  
   \label{degree2} \hypertarget{degree2}{}
  With the assumptions and notations of \cref{settingregular}, assume also that $I$ is generated by elements of the same degree with respect to some $\mathbb{Z}_{>0}$-grading. Then for all $m\geq 1$, the defining ideal of the Rees algebra $\mathcal{R}(I^m)$ is generated in degree at most 2 in the $T$-variables.
 \end{thm}
 
 \begin{proof} 
  It is sufficient to show that any fiber-type relation of the form 
  \begin{equation}
    \label{eqfiber} \hypertarget{eqfiber}{}
      T_{i_1}\cdots T_{i_s} -  T_{j_1}\cdots T_{j_s} \in \mathcal{J}
  \end{equation}   with $s \geq 3$ can be expressed as linear combination of fiber-type relations of degree at most $s-1$.
  First, assume that the following condition holds. \\
  $(\ast)$: Up to relabeling the indices, there exists a generator $g_h \in I$ such that $g_{i_1}g_{i_2}=g_{j_1}g_h$.\\
  In this case, we get 
  $$ T_{i_1}\cdots T_{i_s} -  T_{j_1}\cdots T_{j_s} = T_{i_3}\cdots T_{i_s}(T_{i_1}T_{i_2}-T_{j_1}T_{h}) -  T_{j_1}(T_{j_2}\cdots T_{j_s}-T_hT_{i_3}\cdots T_{i_s}),  $$ and observe that the last summand corresponds to a relation of degree $s-1$, since
  $$  g_{j_2}\cdots g_{j_s} = \frac{ g_{i_1}g_{i_2}\cdots g_{i_s}}{g_{j_1}} = g_{h} g_{i_3} \cdots g_{i_s}. $$ 
  
  Hence, this case is concluded and we may assume that condition $(\ast)$ does not hold for the relation (\ref{eqfiber}).
  Write $g_{i_1}= \gcd(g_{i_1},g_{j_1}) F_1^{p_1} \cdots F_a^{p_a}$ and $g_{j_1}= \gcd(g_{i_1},g_{j_1}) G_1^{q_1} \cdots G_b^{q_b}$, where $F_1,\dots, F_a, G_1,\dots G_b \in \mathcal{F}$. 
  Observe that since the ideal $I^m$ is equigenerated, then $\deg(F_1^{p_1}\cdots F_a^{p_a})= \deg(G_1^{q_1} \cdots G_b^{q_b})$. Moreover, 
  \begin{equation}
    \label{Fl-Glequation} \hypertarget{Fl-Glequation}{}
      F_1^{p_1}\cdots F_a^{p_a} g_{i_2}\cdots g_{i_s} =  G_1^{q_1} \cdots G_b^{q_b} g_{j_2}\cdots g_{j_s}.
  \end{equation}
  Now, necessarily, we can find a $k \geq 2$ so that either $g_{i_k}$ is divisible by some $G_l$ and not divisible by at least one power $F_l^m$ or $g_{j_k}$ is divisible by some $F_l$ and not divisible by at least one power $G_l^m$. 
  Indeed, if exactly $c$ elements among $g_{i_2},\ldots, g_{i_s}$ are divisible by some of the $G_1, \ldots, G_b$ and they are all divisible also by $F_1^m \cdots F_a^m$, then the degree of each $F_l$ in each side of \cref{Fl-Glequation} is at least $cm + p_l$ while the degree of each $G_l$ is at most $cm$. 
  Similarly, if exactly $d$ elements among $g_{j_2},\ldots, g_{j_s}$ are divisible by some of the $F_1, \ldots, F_a$ and they are all divisible also by $G_1^m \cdots G_b^m$, then the degree of each $G_l$ in each side of \cref{Fl-Glequation} is at least $dm + q_l$ while the degree of each $F_l$ is at most $dm$. 
  These conditions cannot be satisfied simultaneously, hence for some $k \geq 2$ there must exist a generator $g_{j_k}$ divisible by some $F_l$ and not divisible by at least one power $G_l^m$.
  In particular, this argument implies that if $\deg(F_1^{p_1}\cdots F_a^{p_a})=1$, then $g_{i_1}=\gcd(g_{i_1},g_{j_1})F_1$, $g_{j_1}=\gcd(g_{i_1},g_{j_1})G_1$ and $g_{j_1}g_{j_k} = g_{i_1}g_h$ where $g_h = \frac{G_1}{F_1}g_{j_k}$ is a generator of $I$.
  Therefore, condition $(\ast)$ is satisfied.

  If $\deg(F_1^{p_1}\cdots F_a^{p_a}) \geq 2$, without loss of generality, possibly relabelling appropriately 
  we may assume that $g_{i_2}$ is divisible by $G_1$ and not divisible by $F_1^m$ and consider the generators of $I^m$, $g_{h_1}:= g_{i_1}\frac{G_1}{H_1}$ and $g_{h_2}:= g_{i_2}\frac{F_1}{G_1}$. It then follows that 
  $$ T_{i_1}\cdots T_{i_s} -  T_{j_1}\cdots T_{j_s}= T_{i_3}\cdots T_{i_s}( T_{i_1}T_{i_2}-T_{h_1}T_{h_2}) + T_{h_1}T_{h_2}T_{i_3}\cdots T_{i_s} - T_{j_1}\cdots T_{j_s}. $$ 
  Since the first summand is a multiple of a fiber-type relation of degree 2, we only need to prove the theorem for the relation $T_{h_1}T_{h_2}T_{i_3}\cdots T_{i_s} - T_{j_1}\cdots T_{j_s}$. Observe that now $g_{h_1}=\gcd(g_{h_1},g_{j_1})F_1^{p_1-1}F_2^{p_2} \cdots F_a^{p_a}$. If this new relation does not satisfy condition $(\ast)$, by iterating the process, in a finite number of steps we get a relation in which the total degree of the monomial in the forms $F_1, \ldots, F_a$ is one, and hence condition $(\ast)$ must finally be satisfied. This concludes the proof.
  \end{proof}

\begin{remark}
  As observed in the introduction, when $\mathcal{F} = \lbrace x_1, \ldots, x_n \rbrace$, the powers $I_{c, \mathcal{F}}^m$ are polymatroidal ideals satisfying the strong exchange property. Hence, by \cite[3.3]{HHVladoiu} and \cite[5.3(b)]{HHpolymatroids} it was already known that these ideals 
  are of fiber type and their fiber-type relations have degree at most two 
  The result \cite[5.3]{HHpolymatroids} is proved using a sorting technique, relying on Gr\"obner bases and monomial orders.
   Also, mapping variables $y_i$  to the $F_i \in \mathcal{F}$ defines a flat map $\, \displaystyle{ K[y_1, \ldots, y_t] \to K[F_1, \ldots, F_t]}$. Since formation of Rees algebras commutes with flat base change (see \cite[1.3]{EHU}), then \cref{powerstarconf} and \cref{degree2} follow from the case of a monomial regular sequence.  However, our direct proof recovers the known results while also giving a new proof in the monomial case that requires less technical machinery. 
\end{remark}

\section{The linear type property of ideals of star configurations}
\label{sectionLinearType} \hypertarget{sectionLinearType}{}

In this section we study under what conditions the ideal of a star configuration is of linear type. Moreover, we give a criterion to determine how this property may fail in the case of linear star configurations. Our first result characterizes the linear type property of star configurations of hypersurfaces. 

 \begin{thm} 
   \label{lineartype} \hypertarget{lineartype}{}
    Suppose that $I_{c, \mathcal{F}} \subseteq R=K[x_1, \ldots, x_n]$ is a star configuration on $\mathcal{F}= \lbrace  F_1, \ldots, F_t \rbrace$. Let $n=\dim R$ and assume that $\mathcal{F}$ contains a regular sequence of length $n$. The following are equivalent:
      \begin{enumerate}[\rm(1)]
        \item $I_{c, \mathcal{F}}$ is of linear type.
        \item $\mu(I_{c, \mathcal{F}}) \leq n$.
        \item $c=2$ and $n=t$, i.e. $\mathcal{F}$ is a regular sequence. 
      \end{enumerate}
 \end{thm}

 \begin{proof}
  Assume that $I_{c, \mathcal{F}}$ is of linear type. Then, $I_{c, \mathcal{F}}$ satisfies the $G_{\infty}$ condition and in particular $\mu(I_{c, \mathcal{F}}) \leq n$. In general, by \cref{generatorsI} $\,\mu(I_{c, \mathcal{F}})= \binom{t}{t-c+1} \geq t \geq n$. Hence, since $2 \leq c \leq n-1$, $\mu(I_{c, \mathcal{F}}) \leq n$ if and only if $c=2$ and $n=t$. In particular, in this case $\mu(I_{c, \mathcal{F}}) = n$.

  Finally, if $c=2$ and $n=t$, we know by \cref{powerstarconf} and \cref{degree2} that $I_{c, \mathcal{F}}$ is of fiber type and the non-linear relations are generated by binomials of degree 2. But, if there was a nonzero relation of degree two of the form $T_{i_1}T_{i_2}-T_{i_3}T_{i_4}$ for distinct indices $ i_1, i_2, i_3,i_4$, setting $I_{c, \mathcal{F}}=(g_1, \ldots, g_n)$, we would have $g_{i_1}g_{i_2}-g_{i_3}g_{i_4}=0$. Since $c=2$, we have that $g_i= \prod_{j \neq i} F_j$ and this implies $F_{i_1}F_{i_2}=F_{i_3}F_{i_4}$ that is a contradiction since $n=t$ and $F_{1}, \ldots, F_{t}$ form a regular sequence. Therefore there are no relations of degree two and $I_{c, \mathcal{F}}$ is of linear type.
 \end{proof}

 \begin{remark} \label{remarkc=2} \hypertarget{remarkc=2}{}
   From the proof of (3) implies (1) it follows that if $c=2$ and all the elements of $\mathcal{F}$ form a regular sequence then $I=I_{2, \mathcal{F}}$ is of linear type. In this case, since $I$ is a perfect ideal of height 2, then the Rees algebra of $I$ is Cohen-Macaulay by \cite[2.6]{HSV}.
 \end{remark}

 It is clear from the definition that localizing an ideal of linear type at any prime ideal produces an ideal of linear type. In the rest of this section we aim to measure how an ideal of star configurations may fail to be of linear type by examining its linear type property locally. For this purpose, we first need to understand how ideals of star configurations behave under localization. 

 \begin{lem} \label{localization} \hypertarget{localization}{}
   Let $I_{c, \mathcal{F}} \subseteq R$ be a star configuration on $\mathcal{F}$ and let $\mathfrak{p}$ be a prime ideal of $R$. Set $\mathcal{F}^{\prime}:= \mathcal{F} \cap \mathfrak{p}$.
   Then 
    $$ (I_{c, \mathcal{F}})_{\mathfrak{p}} = \left\{ \begin{array}{ccc} I_{c, \mathcal{F}^{\prime}}R_{\mathfrak{p}} &\mbox{if } |\mathcal{F}^{\prime}|  > c \\
       \mbox{ complete intersection of height } c &\mbox{if } |\mathcal{F}^{\prime}|  = c \\
       R_{\mathfrak{p}}  &\mbox{if } |\mathcal{F}^{\prime}|  < c. \\
    \end{array}\right.  $$
 \end{lem}

 \begin{proof}
   Finite intersection of ideals commutes with localization, hence 
     $$ (I_{c, \mathcal{F}})_{\mathfrak{p}} = \bigcap_{i_1, \ldots, i_c} (F_{i_1}, \ldots, F_{i_c})_{\mathfrak{p}}.  $$ 
   If some $F_{i_j} \not \in \mathfrak{p}$, the ideal $(F_{i_1}, \ldots, F_{i_c})_{\mathfrak{p}} = R_{\mathfrak{p}}$. If  $|\mathcal{F}^{\prime}|  < c$, this necessarily happens for all the ideals in the intersection and $(I_{c, \mathcal{F}})_{\mathfrak{p}}=R_{\mathfrak{p}}$. 

   If some $F_{i_1}, \ldots, F_{i_s} \in \mathfrak{p}$ and they form a regular sequence in $R$, then they form a regular sequence also in the ring $R_{\mathfrak{p}}$. 
   Hence, in the case $|\mathcal{F}^{\prime}|  = c$, then $(I_{c, \mathcal{F}})_{\mathfrak{p}}=(F_{i_1}, \ldots, F_{i_c})R_{\mathfrak{p}}$ where $\mathcal{F}^{\prime}=\{F_{i_1}, \ldots, F_{i_c}\}$ and it is a complete intersection of height $c$.
   Instead, if $|\mathcal{F}^{\prime}|  > c$, the ideal $(I_{c, \mathcal{F}})_{\mathfrak{p}} = I_{c, \mathcal{F'}} R_{\mathfrak{p}}$ is a star configuration.
 \end{proof}

 Although the proofs of \cref{lineartype} and \cref{localization} work for any star configuration of hypersurfaces, in the rest of this section we restrict to the case when the $F_i$'s are all linear.  
 
 \begin{prop} \label{Gs} \hypertarget{Gs}{} 
   Let $R=K[x_1, \ldots, x_n]$ and let $I_{c, \mathcal{F}}$ be a linear star configuration of height $c$. The following are equivalent:
  \begin{enumerate}[\rm(1)]
     \item $(I_{c, \mathcal{F}})_{\mathfrak{p}}$ is of linear type for every prime $\mathfrak{p} \in \Spec(R)$ with $ \het{\mathfrak{p}} \leq s -1$.
     \item $I_{c, \mathcal{F}}$ satisfies the $G_s$ condition. 
  \end{enumerate}
 \end{prop}

 \begin{proof}
   By \cref{lintypeGinfty}, condition (1) always implies condition (2). Thus assume that $I_{c,\mathcal{F}}$ satisfies the $G_s$ condition. By way of contradiction, say that there exists a prime ideal $\mathfrak{p}$ of height $\leq s-1$ such that $(I_{c,\mathcal{F}})_{\mathfrak{p}}$ is not of linear type. By \cref{localization}, this implies $|\mathcal{F} \cap \mathfrak{p}| > c.$ Call $\mathfrak{q}$ the ideal generated by the elements of $\mathcal{F} \cap \mathfrak{p}$. Clearly $\mathfrak{q}$ is prime since the elements of $\mathcal{F}$ are linear forms and $\mathfrak{q} \subseteq \mathfrak{p}$. Furthermore $\mathcal{F} \cap \mathfrak{q}= \mathcal{F} \cap \mathfrak{p}$ and thus by \cref{localization}, $(I_{c,\mathcal{F}})_{\mathfrak{q}}$ is a star configuration of height $c$. We want to show that $\,\mu((I_{c,\mathcal{F}})_{\mathfrak{q}}) > \het \mathfrak{q}$; since $\het \mathfrak{q} \leq s-1,\,$ this would contradict the assumption that $I_{c,\mathcal{F}}$ satisfies the $G_s$ condition.
   
   Notice that the height of $\mathfrak{q}$ is equal to the length of a maximal regular sequence contained in $\mathcal{F} \cap \mathfrak{p}$ and for this reason $(I_{c,\mathcal{F}})_{\mathfrak{q}}$ satisfies the assumptions of \cref{lineartype}.
   Hence, from \cref{lineartype} it follows immediately that $\,\mu((I_{c,\mathcal{F}})_{\mathfrak{q}}) > \het \mathfrak{q}$ whenever $c \geq 3$. If $c=2$, observe that $\,\mu((I_{c, \mathcal{F}})_{\mathfrak{q}}) = |\mathcal{F} \cap \mathfrak{q}| = |\mathcal{F} \cap \mathfrak{p}|$ and it suffices to show that $\mathcal{F} \cap \mathfrak{p}$ is not a regular sequence. But if this were a regular sequence, $(I_{c,\mathcal{F}})_{\mathfrak{p}}$ would be a star configuration of height 2 generated over a regular sequence, hence it would be of linear type by \cref{remarkc=2}.
 \end{proof} 
  
 We now describe the set of primes at which $I_{c, \mathcal{F}}$ fails to be of linear type. We call this set of primes the \emph{non-linear type locus} of $I_{c, \mathcal{F}}$ and denote it by
  $$ NLT(I_{c, \mathcal{F}}) = \lbrace \mathfrak{p} \in \Spec(R) \mbox{ : } (I_{c, \mathcal{F}})_{\mathfrak{p}} \mbox{ is not of linear type}\rbrace.  $$

 \begin{prop}
   \label{nonlineartypelocus} \hypertarget{nonlineartypelocus}{}
  Let $R=K[x_1, \ldots, x_n]$ and let $I_{c, \mathcal{F}}$ be a linear star configuration of height $c$.
  Then, any set $\,\mathcal{H} \subseteq \mathcal{F}\,$ of cardinality $s \leq n$ that is not a regular sequence generates a prime ideal $\mathfrak{q} \in NLT(I_{c, \mathcal{F}}) \setminus \lbrace (x_1,\ldots, x_n)\rbrace$.
  
  Moreover, if $c=2$ the minimal elements of $NLT(I_{2, \mathcal{F}})$ that are not maximal ideals of $R$ are all of this form.
  If instead $c \geq 3$, then $\,NLT(I_{c, \mathcal{F}}) =\lbrace \mathfrak{p} \in \Spec(R) \mbox{ : } | \mathcal{F} \cap \mathfrak{p} | > c \rbrace $.
 \end{prop}

 \begin{proof}
  First consider a set $\mathcal{H} \subseteq \mathcal{F}$ of cardinality $s \leq n$ that is not a regular sequence. Since $I_{c, \mathcal{F}}$ is a star configuration, clearly $|\mathcal{H}| > c+1$.
  The elements of $\mathcal{H}$ are linear forms, thus they generate a prime ideal $\mathfrak{q}$ of $R$ which is clearly non-maximal since it has depth strictly smaller than $n$. By \cref{localization}, 
    $$\het(\mathfrak{q}) < |\mathcal{H}| \leq |\mathcal{F} \cap \mathfrak{q}| \leq \mu((I_{c, \mathcal{F}})_{\mathfrak{q}}).$$ 
  Hence,  $(I_{c, \mathcal{F}})_{\mathfrak{q}}$ is not of linear type.

  In the case $c=2$, to show that all minimal primes in $NLT(I_{2, \mathcal{F}})$ arise in this way, let $\mathfrak{p}$ be a non-maximal prime ideal of $R$ such that $(I_{2, \mathcal{F}})_{\mathfrak{p}}$ is not of linear type. By \cref{localization}, $(I_{2, \mathcal{F}})_{\mathfrak{p}}$ is the star configuration of height 2 over the set $\mathcal{F} \cap \mathfrak{p}$ in the ring $R_{\mathfrak{p}}$.
  If $|\mathcal{F} \cap \mathfrak{p}| > n$, consider any subset $\mathcal{H} \subseteq (\mathcal{F} \cap \mathfrak{p})$ of cardinality $n$.
  Since $\mathfrak{p}$ is non-maximal, necessarily $\mathcal{H}$ is not a regular sequence. Then the prime ideal $\mathfrak{q}$ generated by the elements of $\mathcal{H}$ is contained in $\mathfrak{p}$ and is such that $(I_{2, \mathcal{F}})_{\mathfrak{q}}$ is not of linear type by the first part of this proof.

  If, instead, $|\mathcal{F} \cap \mathfrak{p}| \leq n$, call $\mathfrak{q}$ the prime ideal generated by all the elements of $\mathcal{F} \cap \mathfrak{p}$. Clearly $\mathfrak{q} \subseteq \mathfrak{p}$ and  $\mathcal{F} \cap \mathfrak{p}= \mathcal{F} \cap \mathfrak{q}$. Notice that $\mathcal{F} \cap \mathfrak{p}$ is not a regular sequence, since by \cref{remarkc=2} a star configuration of height 2 generated over a regular sequence is always of linear type.
  Therefore,
    $$ \het(\mathfrak{q}) < |\mathcal{F} \cap \mathfrak{p}|= |\mathcal{F} \cap \mathfrak{q}| = \mu((I_{2, \mathcal{F}})_{\mathfrak{q}}). $$
  Hence $(I_{2, \mathcal{F}})_{\mathfrak{q}}$ is not of linear type. 

  When $c \geq 3$, choose a non-maximal prime ideal $\mathfrak{p}.$ By \cref{localization}, if $|\mathcal{F} \cap \mathfrak{p}| \leq c$, then $(I_{c, \mathcal{F}})_{\mathfrak{p}}$ is of linear type. Otherwise, if $|\mathcal{F} \cap \mathfrak{p}| > c$, let $\mathfrak{q}$ be the prime ideal generated by the elements of $\mathcal{F} \cap \mathfrak{p}$. Proceeding as in the proof of \cref{Gs}, we can apply \cref{lineartype} to deduce that $(I_{c, \mathcal{F}})_{\mathfrak{q}}$ is not of linear type. Since $\mathfrak{q} \subseteq \mathfrak{p}$, then also $(I_{c, \mathcal{F}})_{\mathfrak{p}}$ is not of linear type.
 \end{proof}

 We next apply the previous result to characterize when $I_{c, \mathcal{F}}$ satisfies the $G_n$ condition.

 \begin{thm}
  \label{Gngeneral} \hypertarget{Gngeneral}{}
    Let $R=K[x_1, \ldots, x_n]$ and let $I_{c, \mathcal{F}}$ be a linear star configuration of height $c$.
  \begin{enumerate}[\rm(1)]
    \item The ideal $I_{c, \mathcal{F}}$ satisfies the $G_n$ condition. 
 
    \item Every subset of $\mathcal{F}$ of cardinality $n$ is a regular sequence and $c \in \{ 2, n-1 \}.$
  \end{enumerate}
 \end{thm}

 \begin{proof}
   Recall that by \cref{Gs}, $I_{c, \mathcal{F}}$ satisfies the $G_n$ condition if and only if $NLT(I_{c, \mathcal{F}})= \lbrace (x_1, \ldots, x_n) \rbrace$.

   First assume that there exists one subset of cardinality $n$ of $\mathcal{F}$ that is not a regular sequence. By \cref{nonlineartypelocus}, it then follows that $NLT(I_{c, \mathcal{F}})$ contains some non-maximal prime ideal of $R$. Thus $I_{c, \mathcal{F}}$ does not satisfy the $G_n$ condition.

   Hence, we can assume that every subset of $\mathcal{F}$ of cardinality $n$ is a regular sequence. 
   If $c=2$, the conclusion follows by \cref{nonlineartypelocus}. 
   If $c=n-1$, let $\mathfrak{p}$ be a non-maximal prime ideal of $R$. We show that in this case $|\mathcal{F} \cap \mathfrak{p}| \leq c$ and conclude that $(I_{c, \mathcal{F}})_{\mathfrak{p}}$ is of linear type by \cref{localization}. Indeed, $|\mathcal{F} \cap \mathfrak{p}| > c=n-1$ if and only if at least $n$ distinct elements of $\mathcal{F}$ are in $\mathfrak{p}$. But, since $I_{n-1, \mathcal{F}}$ is a star configuration, any $n$ distinct elements of $\mathcal{F}$ form a regular sequence and hence they cannot be contained in a non-maximal prime ideal of $R$.

   Finally, if $2 < c < n-1$, 
   let $\mathcal{H}$ be a subset of $\mathcal{F}$ of cardinality $n-1$, and let $\mathfrak{p}$ be the prime ideal generated by the elements of $\mathcal{H}$. Since every subset of $\mathcal{F}$ of cardinality $n$ is a regular sequence, it follows that any $F_i \in \mathcal{F} \setminus \mathcal{H}$ is regular modulo $\mathfrak{p}$, hence is not in $\mathfrak{p}$. Therefore,  $\mathcal{F} \cap \mathfrak{p} = \mathcal{H}$ has cardinality $n-1 > c $, hence the conclusion follows from \cref{nonlineartypelocus}.
 \end{proof}

 \begin{cor}
  Let $R=K[x_1, \ldots, x_n]$ and assume that any subset of $\mathcal{F}$ of cardinality $n$ is a regular sequence. Let $I_{2, \mathcal{F}}$ be a linear star configuration of height $2$. Then $I$ is of fiber type and the defining ideal of $\mathcal{R}(I)$ is given by 
    $$\mathcal{J}= \mathcal{L}+I_n(B(M)),$$ 
  where $B(M)$ is the ideal of maximal minors of the Jacobian dual of $M$. Moreover, $\mathcal{R}(I)$ and $F(I)$ are Cohen-Macaulay.  
 \end{cor} 

 \begin{proof}
   From \cref{Gngeneral} it follows that $I$ satisfies the $G_n$ condition. Then the thesis follows from \cref{morey-ulrich}. 
 \end{proof}

 In the next section we give a more accurate description of the Rees algebra of linear star configurations of height two, providing an explicit generating set for the ideal of maximal minors of the Jacobian dual. This also allows us to determine the defining ideal in the case when the $G_n$ condition is not satisfied. For this purpose, it will be convenient to better control the subsets of regular sequences contained in $\mathcal{F}$. The following lemma will be useful. 

 \begin{lem} \label{linearsettinglemma} \hypertarget{linearsettinglemma}{}
   Let $K$ be an infinite field and let $I_{c,\mathcal{F}}$ be a linear star configuration of height $c$ defined over the set $\mathcal{F}=\lbrace F_1, \ldots, F_{t} \rbrace \subseteq K[y_1,\ldots, y_d]$. Assume that the maximal regular sequence contained in $\mathcal{F}$ has length $n$. 
   Then, after renaming the variables of the ring, we can always assume
     $$\mathcal{F}= \lbrace F_1, \ldots, F_{t} \rbrace = \lbrace x_1, \ldots, x_n, L_{1}, \ldots, L_{r} \rbrace \subseteq K[x_1,\ldots, x_n]$$ 
   where $L_{1}, \ldots, L_{r} \in (x_1, \ldots, x_n)$ are linear forms.
 \end{lem}

\begin{proof}
  By relabeling the indices assume that $ F_1, \ldots, F_{n} $ is a regular sequence of maximal length contained in $\mathcal{F}$.
  After a linear change of variables, this regular sequence of linear forms can be always expressed as $ x_1, \ldots, x_n $. 
  If some linear form $F_j$ with $j > n$ has a monomial of the form $y_k$, with $y_k$ distinct from $ x_1, \ldots, x_n $, then clearly $F_j,x_1, \ldots, x_n$ is a regular sequence of length $n+1$ contradicting the hypothesis. It follows that $F_{n+1}, \ldots, F_t$ must be contained in $(x_1, \ldots, x_n)$.
\end{proof}


 Thanks to \cref{linearsettinglemma}, whenever the $F_i$'s are linear forms  we can always reduce to the following setting.  
  
 \begin{setting} \label{linearsetting} \hypertarget{linearsetting}{}
   Let $K$ be an infinite field and let $R=K[x_1,\ldots, x_n]$. Assume that $$\mathcal{F}= \lbrace F_1, \ldots, F_{t} \rbrace = \lbrace x_1, \ldots, x_n, L_{1}, \ldots, L_{r} \rbrace \subseteq R $$ where $L_i \coloneq \sum_{j=1}^n u_{ij}x_j$ with $u_{ij} \in K$ and $t=n+r$. Let $U$ denote the $n \times r$ matrix on the elements $u_{ij}$ and let $I_{c, \mathcal{F}} \subseteq R$ be a star configuration on $\mathcal{F}$.
 \end{setting}

\begin{prop} \label{sregseq-minorsU} \hypertarget{sregseq-minorsU}{}
 Fix $2 \leq s \leq n$. With the assumptions and notations of \cref{linearsetting}, the following conditions are equivalent:
  \begin{enumerate}[\rm(1)]
   \item Any subset of $\mathcal{F}$ of cardinality $s$ is a regular sequence.
   \item For every $1 \leq h \leq \min \lbrace r,s \rbrace$, all the submatrices of the matrix $U$ of size $(h+n-s) \times h$ have maximal rank.
  \end{enumerate}
\end{prop}

\begin{proof}
Recall that a finite set of linear forms in $R$ is a regular sequence if and only if those forms are linearly independent over the base field $K$. This is equivalent to have that the matrix expressing their coefficients as a function of the variables $x_1, \ldots, x_n$ has maximal rank.

After fixing $s$, assume one submatrix $V$ of $U$ of size $(h+n-s) \times h$ has not maximal rank. For simplicity, up to permuting rows and columns, we may assume this to be the matrix obtained considering the first $h+n-s$ rows and the first $h$ columns of $U$ for some $h \leq \min \lbrace r,s \rbrace$. If $h=s$, this implies that the linear forms $L_1, \ldots, L_h $ are not linearly independent over $K$, thus not a regular sequence. Otherwise, if $h < s$, consider the set $\lbrace L_1, \ldots, L_h, x_{n+h-s+1}, \ldots, x_{n} \rbrace$. Going modulo the regular sequence $\lbrace x_{n+h-s+1}, \ldots, x_{n} \rbrace$, this set reduces to the set of linear forms $\lbrace \overline{L_1}, \ldots, \overline{L_h} \rbrace $. The matrix of coefficients of these linear forms with respect to $x_1, \ldots, x_{n+h-s}$ is exactly the matrix $V$. It follows that $ \overline{L_1}, \ldots, \overline{L_h} $ are not linearly independent, hence not a regular sequence. Therefore $\lbrace L_1, \ldots, L_h, x_{n+h-s+1}, \ldots, x_{n} \rbrace$ is not a regular sequence and has cardinality $s$.
 
Conversely, assume that condition 2 is satisfied. Consider a subset $\mathcal{H}$ of $\mathcal{F}$ of cardinality $s $. Possibly going modulo the forms of type $x_1, \ldots, x_n$ belonging to $\mathcal{H}$, we reduce to consider a set of linear forms $\lbrace \overline{L_{i_1}}, \ldots, \overline{L_{i_h}} \rbrace $ with $h \leq s$  and whose coefficients with respect to $\overline{x_1}, \ldots, \overline{x_n}$ form a submatrix of $U$ of size $(h+n-s) \times h$. It follows that those forms are linearly independent and hence a regular sequence. Thus also the elements of $\mathcal{H}$ form a regular sequence.
\end{proof}

\begin{remark}
\label{minors-regseq} \hypertarget{minors-regseq}{}
The proof of the \cref{sregseq-minorsU} shows that there is a one-to-one correspondence between the subsets of $\mathcal{F}$ of cardinality $s \leq n$ and the submatrices of $U$ of size $(h+n-s) \times h$. One of such subsets of $\mathcal{F}$ is a regular sequence if and only if the corresponding matrix has maximal rank.
 \end{remark}

\section{The maximal minors of the Jacobian dual of linear star configurations of height two}
 \label{sectionJacobiandual} \hypertarget{sectionJacobiandual}{}

The main goal of this section is to determine a minimal generating set for the ideal of maximal minors of the Jacobian dual of a presentation of $I_{2,\mathcal{F}}$ (see \cref{minorsJacDual}). In the case when $I_{2,\mathcal{F}}$ satisfies $G_n$, by \cref{morey-ulrich} these are the non-linear equations defining the Rees algebra of $I_{2,\mathcal{F}}$. In the next section, we will use 
\cref{minorsJacDual} to identify the non-linear equations of the Rees algebra also when the $G_n$ condition is not satisfied (see \cref{orlikterao}).

\begin{prop}
\label{jacobheight2} \hypertarget{jacobheight2}{}
With the assumptions and notations of \cref{linearsetting}, let $c=2$. There exists a presentation matrix $M$ of $I_{2,\mathcal{F}}$ whose Jacobian dual can be expressed in the following form:
  \begin{equation}
    \label{jacobiandual} \hypertarget{jacobiandual}{}
     \footnotesize B(M)= \bmatrix 
T_1 & 0 & \ldots & 0 &   u_{11}T_{n+1} &  u_{21}T_{n+2} & \ldots &  u_{r1}T_{n+r} \\ 
-T_2 & T_2 & \ldots & 0 &   u_{12}T_{n+1} &  u_{22}T_{n+2} & \ldots &  u_{r2}T_{n+r}\\  
0 & -T_3 & \ldots & \vdots & \vdots & \vdots & \ldots & \vdots  \\  
\vdots & 0 & \ldots & \vdots & \vdots & \vdots & \ldots & \vdots \\  
\vdots & \vdots & \ldots &  0 & \vdots & \vdots & \ldots & \vdots \\  
0 & 0 & \ldots &  T_{n-1} &  u_{1,n-1}T_{n+1} & u_{2,n-1}T_{n+2} & \ldots &  u_{r,n-1}T_{n+r} \\  
0 & 0 & \ldots &  -T_n &  (u_{1n}T_{n+1}-T_n) & (u_{2n}T_{n+2}-T_n) & \ldots &  (u_{rn}T_{n+r}-T_n)\\  
 \endbmatrix.  
\end{equation}
\end{prop}

\begin{proof}
Since $I_{2,\mathcal{F}}$ is perfect of height two, by the Hilbert-Burch Theorem a presentation matrix of $I_{2,\mathcal{F}}$ is 
\begin{equation}
\label{presentationmatrix} \hypertarget{presentationmatrix}{}
   \footnotesize M= \bmatrix 
F_1 & 0 & \ldots & 0   \\ 
-F_2 & F_2 & \ldots & \vdots   \\  
0 & -F_3 & \ldots & \vdots  \\  
\vdots & 0 & \ldots & 0  \\  
\vdots & \vdots & \ldots &  F_{t-1}  \\  
\vdots & \vdots & \ldots &  -F_{t}  \\  
 \endbmatrix. 
\end{equation}   
Recall that $\mathcal{F}= \lbrace F_1, \ldots, F_t \rbrace = \lbrace x_1, \ldots, x_n, L_1, \ldots, L_r \rbrace$ where $L_i=\sum_{j=1}^n u_{ij}x_j$.
 Therefore, the Jacobian dual $B(M)$ is equal to $$ \footnotesize \bmatrix 
T_1 & 0 & \ldots & 0 &  - u_{11}T_{n+1} &  (u_{11}T_{n+1} - u_{21}T_{n+2}) & \ldots &  (u_{r-1,1}T_{n+r-1} - u_{r1}T_{n+r}) \\ 
-T_2 & T_2 & \ldots & 0 &  - u_{12}T_{n+1} &  (u_{12}T_{n+1} - u_{22}T_{n+2}) & \ldots &  (u_{r-1,2}T_{n+r-1} - u_{r2}T_{n+r})\\  
0 & -T_3 & \ldots & \vdots & \vdots & \vdots & \ldots & \vdots  \\  
\vdots & 0 & \ldots & \vdots & \vdots & \vdots & \ldots & \vdots \\  
\vdots & \vdots & \ldots &  0 & \vdots & \vdots & \ldots & \vdots \\  
0 & 0 & \ldots &  T_{n-1} & - u_{1,n-1}T_{n+1} & (u_{1,n-1}T_{n+1} - u_{2,n-1}T_{n+2}) & \ldots &  (u_{r-1,n-1}T_{n+r-1} - u_{r,n-1}T_{n+r}) \\  
0 & 0 & \ldots &  -T_n & (T_n- u_{1n}T_{n+1}) & (u_{1n}T_{n+1} - u_{2n}T_{n+2}) & \ldots &  (u_{r-1,n}T_{n+r-1} - u_{rn}T_{n+r})\\  \endbmatrix.  $$ \normalsize
 By column operations this matrix can be reduced to the form given in \cref{jacobiandual}.
\end{proof}


 
 





Notice that the Jacobian dual of $I_{2,\mathcal{F}}$ is of size $n \times (t-1)$. To compute the ideal of maximal minors of the Jacobian dual we introduce the following notations. 

\begin{definition}
\label{minorsU} \hypertarget{minorsU}{}
Assume $r = t-n \geq 1$. As in \cref{linearsetting}, for $i=1, \ldots ,r$ consider the linear forms $L_i \coloneq \sum_{k=1}^n u_{ik}x_k$ with $u_{ik} \in K$ and let $U$ be the $n \times r$ matrix on the elements $u_{ik}$.

Consider a set of indexes $\chi \subseteq \lbrace 1, \ldots, t \rbrace$ such that $|\chi|= r$. Write 
$$ \chi = \lbrace i_1, \ldots, i_h, j_1, \ldots, j_{r-h} \rbrace$$ with $i_1, \ldots, i_h \leq n$ and $j_1, \ldots, j_{r-h} \geq n+1$. Denote by $U_{\chi}$ the $h \times h$ minor of $U$ defined by taking the rows $i_1, \ldots, i_h$ and removing the columns $j_1-n, \ldots, j_{r-h}-n$. 
In the case $h=0$, we set $U_{\chi}\coloneq1$. 
\end{definition}


\begin{definition}
\label{minorsB} \hypertarget{minorsB}{}
Consider the polynomial ring $K[T_1, \ldots, T_t]$. Suppose $r=t-n \geq 1$ and define the following polynomials.
Given $\Theta \subseteq \lbrace 1, \ldots, t \rbrace$ such that $|\Theta|= r-1$,  set
$ \lbrace 1, \ldots, t \rbrace \setminus \Theta = \lbrace k_1, \ldots, k_{n+1} \rbrace,$ where $k_i < k_{i+1}$ for every $i= 1, \ldots, n$. Define
$$ m_{\Theta}:= \sum_{i=1}^{n+1} (-1)^{\alpha(\Theta,k_i)} \left( \dfrac{T_{k_1} \cdots T_{k_{n+1}}}{T_{k_i}} \right) U_{\Theta \cup \lbrace k_i \rbrace} \in K[T_1, \ldots, T_t],  $$ where $U_{\Theta \cup \lbrace k_i \rbrace}$ is defined as in \cref{minorsU}. 
The exponent $\alpha(\Theta,k_i)$ is obtained as follows. 
Let $h < n+1$ be the integer such that $k_h \leq n$ and $k_{h+1} > n$. Then 
$$ \alpha(\Theta,k_i) =  \left\{ \begin{array}{cc}
   n-h +i - k_i  & \mbox{ if } i \leq h \\
  n+i   & \mbox{ if } i > h.
\end{array} \right. $$
\end{definition}

Notice that none of the monomials of $m_{\Theta}$ is divisible by any variable $T_j$ for $j \in \Theta$ and that $m_{\Theta} \in (T_{k_i}, T_{k_l}) $ for any $k_i,k_l $.

We now state our main theorem about the ideal of maximal minors of the Jacobian dual matrix of $I_{2,c}(\mathcal{F})$.

\begin{thm}
\label{minorsJacDual} \hypertarget{minorsJacDual}{}
With the assumptions of \cref{linearsetting}, let $c=2$, and let $B$ be the Jacobian dual matrix for $I_{2,\mathcal{F}}$ described by \cref{jacobiandual}. Moreover, assume that $r = t-n \geq 1$.  The ideal $I_n(B)$ of the maximal minors of $B$ is minimally generated by all the polynomials $m_{\Theta}$ defined in \cref{minorsB} such that $\,n \not \in \Theta$. \end{thm}

Before proving the theorem, we need a technical definition and a lemma that give us control on the minors of the submatrices of $B$ containing the last $r$ columns.

\begin{definition}
 \label{moreminors} \hypertarget{moreminors}{} 
 Adopt the same notation of \cref{minorsB}. Let $B$ be the Jacobian dual matrix for $I_{2,\mathcal{F}}$ described by \cref{jacobiandual} and denote its columns by $A_1, \ldots, A_{n-1}, C_1, \ldots, C_r$. 
 Consider a set of indices $\Theta \subseteq \lbrace 1, \ldots, n-1 \rbrace$ such that $|\Theta|=r-1.$ \\
 If $1 \not \in \Theta$ write $ \Theta = \Theta_1 \cup \ldots \cup \Theta_s \,$ such that for every $\,i=1, \ldots, s$:
  \begin{itemize}
    \item $ \Theta_i = \lbrace k_i, k_i+1, \ldots, k_i+l_i-1 \rbrace $ contains $l_i$ consecutive indexes.
    \item $k_i + l_i -1 < k_{i+1} -1$.
  \end{itemize}
 In particular, for every $i$, $k_i-1 \not \in \Theta$.
 If $\,1 \in \Theta$, write in the same way $\, \Theta = \Theta_0 \cup \Theta_1 \cup \ldots \cup \Theta_s $ with $1 \in \Theta_0$.  

 For $s=0$ (i.e. $\Theta = \lbrace 1, \ldots, r-1 \rbrace$), let $p_{\Theta}^{(1,1)}$ denote the maximal minor of the submatrix of $B$ computed by only removing the first $r-1$ columns.
 If $s \geq 1$, for $i=1, \ldots, s$ and $j= 1, \ldots, l_i$ denote by $p_{\Theta}^{(i,j)}$ the maximal minor of the matrix obtained from $B$ by removing all the columns $A_k$ for $k \in \Theta$ and by doing the following column operations: 
  \begin{itemize}
    \item for $ 1 \leq h < i$ replace the column $A_{k_h-1}$ with the sum of consecutive columns $ A_{k_h-1}+ A_{k_h} + \ldots + A_{k_h+l_h-1}   $.
    \item replace the column $A_{k_i-1}$ with the sum of consecutive columns $ A_{k_i-1}+ A_{k_i} + \ldots + A_{k_i+j-2}   $.
  \end{itemize}
 Only for $i=s$, we define in an analogous way another minor $p_{\Theta}^{(s,l_s+1)}$. 
\end{definition}

In the proof of \cref{minorsJacDual} we show that the minor $p_{\Theta}^{(s,l_s+1)}$ coincides with the polynomial $m_{\Theta}$, and that the ideal generated by all the polynomials $m_{\Theta}$ coincides with the ideal generated by all the minors $p_{\Theta}^{(1,1)}$. A key fact is that the minors $p_{\Theta}^{(i,j)}$ define a sequence, where every element is obtained from the previous one by replacing one column of the corresponding submatrix with its sum with a column of $B$ excluded from such submatrix.
 We rename this sequence of minors as
 $$  p_{\Theta}^{(1)},p_{\Theta}^{(2)}, \ldots, p_{\Theta}^{(e)} := p_{\Theta}^{(1,1)}, \ldots, p_{\Theta}^{(1,l_1)}, p_{\Theta}^{(2,1)}, \ldots, p_{\Theta}^{(2, l_2)}, \ldots, p_{\Theta}^{(s,l_s)}, p_{\Theta}^{(s,1)}, \ldots, p_{\Theta}^{(s,l_s+1)}. $$
 The following lemma provides a formula relating different elements of this sequence of minors, which allows us to prove \cref{minorsJacDual} by induction. To help the reader dealing with the technicality of our argument, in \cref{example1} we show how to apply the lemma for small values of $r$.

\begin{lem}
\label{lemmaminors} \hypertarget{lemmaminors}{}
Adopt the same notations as in \cref{minorsB} and \cref{moreminors}. For $i=1, \ldots, s$ and $j= 1, \ldots, l_i$, set $\,\Theta(i,j) \coloneq \Theta \setminus \{ k_i+j-1 \} \cup \{ k_i - 1 \}$.
For $h=1, \ldots, e-1$ we have 
 \begin{equation}
  \label{minorsformula} \hypertarget{minorsformula}{}
   p_{\Theta}^{(h)} \coloneq p_{\Theta}^{(i,j)} =  p_{\Theta}^{(h+1)} - p_{\Theta(i,j)}^{(h-j+1)}. 
  \end{equation}
\end{lem}
 
\begin{proof}
Recalling that the matrix $B$ can be expressed as in \cref{jacobiandual}, up to permuting columns, by definition $p_{\Theta}^{(i,j)}$ is the determinant of a matrix of the form 
  $$ \footnotesize  \bmatrix 
     0    &  &\\
     \vdots     &  &\\
     T_{k_i -1} &   & \\
     0    &  &\\
     \vdots     &  B' \\
     -T_{k_i-1+j} & & \\
     0     &  &\\
     \vdots  & &\\
     0    &  &
   \endbmatrix.  $$
 Thus we can write $ p_{\Theta}^{(h)}= p_{\Theta}^{(i,j)} = (-1)^{k_i} T_{k_i -1} M_1 - (-1)^{k_i+j} T_{k_i -1+j} M_2 $, where $M_1$ and $M_2$ are minors of the matrix $B'$. Similarly, by definition $ p_{\Theta}^{(h+1)} $ is the determinant of the same matrix, where the first column is replaced by its sum with the column of $B$ containing the two variables $T_{k_i-1+j}, -T_{k_i+j}$.
 Hence it has the form
    $ p_{\Theta}^{(h+1)}=  (-1)^{k_i} T_{k_i -1} M_1 - (-1)^{k_i+j+1} T_{k_i +j} M_3 $ for some minor $M_3$ of $B'$. It follows that 
  $$  p_{\Theta}^{(h)}=  p_{\Theta}^{(h+1)} - (-1)^{k_i+j} (T_{k_i +j} M_3 + T_{k_i -1+j} M_2). $$
 We have to show that $ p_{\Theta(i,j)}^{(h-j+1)} = (-1)^{k_i+j} (T_{k_i -1+j} M_2 + T_{k_i +j} M_3) $. The second term of the equality is clearly the determinant of the matrix 
  $$ \footnotesize  \bmatrix 
     0    &  &\\
     \vdots     &  &\\
     T_{k_i -1+j} &   & \\
     -T_{k_i+j} & B' \\
     0     &  &\\
     \vdots  & &\\
      0    &  &
  \endbmatrix.  $$
 The first column of this matrix is equal to the column $A_{k_i -1+j}$ of $B$. The remaining columns are obtained by performing the operations described by \cref{moreminors} on the set $\Theta$ to obtain the minor $p_{\Theta}^{(i,j)}$. Now, notice that all the indexes smaller than $k_i -1$ are in $\Theta$  if and only if they are in the set $\Theta(i,j)=\Theta \setminus \{ k_i+j-1 \} \cup \{ k_i - 1 \}$. Then, by performing $h-j$ of the operations described in \cref{moreminors} on the set $ \Theta(i,j)$ we obtain the same matrix as above. Hence, the determinant of the matrix above coincides with $ p_{\Theta(i,j)}^{(h-j+1)}$.
  \end{proof}
  
 \begin{example}
  \label{example1} \hypertarget{example1}{}
  Adopt the same notation as in \cref{moreminors} and \cref{lemmaminors}.
  \begin{itemize}
    \item In the case when $r=1$, $B$ has only one maximal minor. In the proof of \cref{minorsJacDual}, this minor will be shown to be equal to $m_{\Theta}$ with $\Theta = \emptyset$.
      \item In the case when $r=2$, following the notation of \cref{moreminors} and \cref{lemmaminors}, we deal with sets $\Theta= \lbrace i \rbrace$ with $1 \leq i \leq n-1$. If $i=1$, we have $s=0$ and $p_{\Theta}^{(1,1)}= m_{\Theta}$.
      For $i \geq 2$, $s=1$ and the sequence corresponding to $\Theta$ is $ p_{\Theta}^{(1,1)}, p_{\Theta}^{(1,2)},  $ where $p_{\Theta}^{(1,1)}$ is the minor of $B$ obtained removing the column $A_i$ and $p_{\Theta}^{(1,2)}$ is the minor obtained by removing the column $A_i$ and replacing $A_{i-1}$ by $A_{i-1} + A_i$. This second minor is equal to $m_{\Theta}$. \cref{lemmaminors} gives
        $$ p_{\Theta}^{(1,1)}= p_{\Theta}^{(1,2)} - p_{\Theta'}^{(1,1)},  $$ 
      with $\Theta'= \lbrace i-1 \rbrace$. Inductively this shows that $p_{\Theta}^{(1,1)}$ is in the ideal generated by the minors of the form $m_{\lbrace j \rbrace}$ for $j \leq i$.
   \item Consider also the case when $r=3$. Here the sequences correspond to sets $\Theta= \lbrace i,j \rbrace$ with $1 \leq i < j \leq n-1$. Again we have $p_{\lbrace 1,2  \rbrace} = m_{\lbrace 1,2  \rbrace}$. Then we have to describe three possible cases: $ \lbrace 1,j \rbrace $ with $j \geq 3$, $ \lbrace i,i+1 \rbrace $, and $ \lbrace i,j \rbrace $ with $i-j \geq 2 $.
   In the first case $s=1, l_1=1$, and similarly to the case $r=2$ we get 
     $$ p_{\lbrace 1,j \rbrace}^{(1,1)}= p_{\lbrace 1,j \rbrace}^{(1,2)} - p_{\lbrace 1,j-1 \rbrace}^{(1,1)}, $$ and $p_{\lbrace 1,j \rbrace}^{(1,2)}= m_{\lbrace 1,j \rbrace}$.
   For $\Theta= \lbrace i,i+1 \rbrace$ we find $s=1, l_1=2$. The minor $p_{\Theta}^{(1,1)}$ is obtained by removing columns $A_i, A_{i+1}$, $p_{\Theta}^{(1,2)}$ is obtained by also replacing the column $A_{i-1}$ by $A_{i-1} + A_i$, and $p_{\Theta}^{(1,3)}= m_{\lbrace i,i+1 \rbrace}$ is obtained replacing the column $A_{i-1}$ by $A_{i-1} + A_i+A_{i+1}$. \cref{lemmaminors} gives 
     $$ p_{\lbrace i,i+1 \rbrace}^{(1,1)}= p_{\lbrace i,i+1 \rbrace}^{(1,2)} - p_{\lbrace i-1,i+1 \rbrace}^{(1,1)}= (p_{\lbrace i,i+1 \rbrace}^{(1,3)} - p_{\lbrace i-1,i \rbrace}^{(1,1)}) - p_{\lbrace i-1,i+1 \rbrace}^{(1,1)}. $$
   In the case $\Theta= \lbrace i,j \rbrace$ with $i > 1$, $i-j \geq 2$, we have $s=2,l_1=1,l_2=1$. Here we have $p_{\Theta}^{(1,1)}$, $p_{\Theta}^{(2,1)}$, $p_{\Theta}^{(2,2)} = m_{\Theta}$ that are obtained subsequently by first removing columns $A_i,A_j$, then replacing $A_{i-1}$ by $A_{i-1} + A_i$, and finally replacing also $A_{j-1}$ by $A_{j-1} + A_j$. By \cref{lemmaminors} $$ p_{\lbrace i,j \rbrace}^{(1,1)}= p_{\lbrace i,j \rbrace}^{(2,1)} - p_{\lbrace i-1,j \rbrace}^{(1,1)}= (p_{\lbrace i,j \rbrace}^{(2,2)} - p_{\lbrace i,j-1 \rbrace}^{(2)}) - p_{\lbrace i-1,j \rbrace}^{(1,1)}, $$ where the notation $(2)$ stands for $(1,2)$ if $j-1= i + 1$ and for $(2,1)$ if $j-1 > i+1$.  Also in this case it follows that each $p_{\Theta}^{(1,1)}$ is in the ideal generated by the polynomials $m_{\Theta}$. Indeed one can combine all the previous formulas and use inductively the fact that the second term of each new equality corresponds to a set $\Theta' $ containing smaller indexes.
 \end{itemize}
\end{example}
  
 We are now ready to prove \cref{minorsJacDual}.


\begin{proof}(of \cref{minorsJacDual}).
 As in \cref{moreminors} denote the columns of $B$ by $A_1, \ldots, A_{n-1}, C_1, \ldots$, $C_r$.
 Given $\Psi \subseteq \lbrace A_1, \ldots, A_{n-1}, C_1, \ldots, C_r \rbrace$ such that $|\Psi|= r-1$, denote by $p_{\Psi}$ the $n \times n$ minor of $B$ obtained by removing all the columns contained in $\Psi$.
 Observe that we need to prove that the ideal $$I_n(B)=(p_{\Psi} \mbox{ : } \Psi \subseteq \lbrace A_1, \ldots, A_{n-1}, C_1, \ldots, C_r \rbrace, \mbox{  } |\Psi|= r-1 )$$ is equal to the ideal $$ (m_{\Theta} \mbox{ : } \Theta \subseteq \lbrace 1, \ldots, n-1, n+1, \ldots, t \rbrace, \mbox{  } |\Theta|= r-1 ). $$
 We fix $n$ and work by induction on $r$. If $r=1$, the matrix $B$ described in \cref{jacobiandual} reduces to the form
    $$ \footnotesize  \bmatrix 
      T_1 & 0 & 0 & \ldots & 0 &   u_{11}T_{n+1}  \\ 
     -T_2 & T_2 & 0 & \ldots & 0 &   u_{12}T_{n+1} \\  
      0 & -T_3 & T_3 &\ldots & \vdots & \vdots  \\  
      \vdots & 0 & -T_4 & \ldots & \vdots & \vdots  \\  
      \vdots & \vdots & 0 & \ldots &  0 & \vdots  \\  
      \vdots & \vdots & \ldots & \vdots & T_{n-1} &  u_{1,n-1}T_{n+1}  \\  
      0 & 0 & 0 & \ldots &  -T_n &  (u_{1n}T_{n+1}-T_n) \\  
     \endbmatrix.  $$
 A quick computation by induction on $n$ shows that for every $n$ the determinant of this matrix is equal to
 $$ m_{\Theta}=  -T_1 \cdots T_n+\sum_{i=1}^{n} \left( \dfrac{T_{1} \cdots T_{n}}{T_{i}} \right) u_{1i} $$ which, according to \cref{minorsB}, corresponds to the set $\Theta = \emptyset$.
 
 Hence, we assume that the result is true for $r-1$ and we prove it for some $1 < r < n$. The case $r \geq n$ will be considered later. 
 
 Consider the matrix in \cref{jacobiandual} and take the submatrix $B^{\prime}$ obtained by eliminating one column $C_j$ with $j \in \lbrace 1, \ldots, r \rbrace $. The ideal of maximal minors of $B^{\prime}$ is contained in $I_n(B)$ and its generators are also generators of $I_n(B)$. By the inductive hypothesis we have 
 $$  I_n(B^{\prime})= (m_{\Theta^{\prime}} \mbox{ : } \Theta^{\prime} \subseteq \lbrace 1, \ldots, n-1, n+1, \ldots, t \rbrace \setminus \lbrace n+j \rbrace, \mbox{  } |\Theta^{\prime}|= r-2)    $$ as ideal of the polynomial ring $K[T_1, \ldots, \widehat{T_{n+j}}, \ldots, T_t]$. Following the notation of \cref{minorsB} and working back in the polynomial ring $K[T_1, \ldots, T_t]$ we observe that each of such $m_{\Theta^{\prime}}$ coincides with $m_{\Theta}$ with $\Theta \coloneq \Theta^{\prime} \cup  \lbrace n+j \rbrace$. Since the same argument can be applied to any $j \in \lbrace 1, \ldots, r \rbrace $, we reduce to considering only the minors of $B$ for submatrices containing all the last $r$ columns $C_1, \ldots, C_r$. In particular we have to show that 
 $$ (p_{\Psi} \mbox{ : } \Psi \subseteq \lbrace A_1, \ldots, A_{n-1} \rbrace, \mbox{  } |\Psi|= r-1 ) = (m_{\Theta} \mbox{ : } \Theta \subseteq \lbrace 1, \ldots, n-1 \rbrace, \mbox{  } |\Theta|= r-1 ).  $$
 
 Consider first the submatrix obtained from $B$ by deleting the first $r-1$ columns. This matrix is equal to
   $$ \footnotesize B^{\star}:= \bmatrix 
      0 & 0 & \ldots & 0 &   u_{11}T_{n+1} &   \ldots &  u_{r1}T_{t} \\ 
      \vdots & \vdots & \ldots & \vdots & \vdots &  \ldots & \vdots \\  
      T_{r} & 0 & \ldots & \vdots & \vdots &  \ldots & \vdots  \\ 
     -T_{r+1} & T_{r+1} & \ldots & \vdots & \vdots &  \ldots & \vdots  \\ 
     0 & -T_{r+2} & \ldots & \vdots & \vdots &  \ldots & \vdots  \\    
     \vdots & \vdots & \ldots &  0 & \vdots &  \ldots & \vdots \\  
     0 & 0 & \ldots &  T_{n-1} &  u_{1,n-1}T_{n+1} &  \ldots &  u_{r,n-1}T_{t} \\  
     0 & 0 & \ldots &  -T_n &  (u_{1n}T_{n+1}-T_n) &  \ldots &  (u_{rn}T_{t}-T_n)\\  
   \endbmatrix.  $$
 One can check by induction on $n$ that its determinant is equal to
   $$ p_{\lbrace A_1, \ldots, A_{r-1} \rbrace}= (-1)^{r} \left[ \sum_{i=r}^{t} (-1)^{\alpha_i} \left( \dfrac{T_{r} \cdots T_{t}}{T_{i}} \right) U_{\Theta \cup \lbrace i \rbrace}  \right] = m_{\Theta^{\star}} ,  $$ 
 where $\Theta^{\star}:= \lbrace 1, \ldots, r-1 \rbrace$ and $\alpha_i= \max\lbrace i-n-1,0 \rbrace = \alpha(\Theta^{\star}, i)$ as in \cref{minorsB}.
 
 Consider now an arbitrary set of indices $\Theta \subseteq \lbrace 1, \ldots, n-1 \rbrace$ such that $|\Theta|=r-1,$ and $\Theta \neq \Theta^{\star}$. Using the notation of \cref{moreminors}, we want to show that $m_{\Theta} = p_{\Theta}^{(s, l_s+1)}$ and therefore is in the ideal $I_n(B)$. Write $\lbrace 1, \ldots, t \rbrace \setminus \Theta = \lbrace  k_1, \ldots, k_{n+1} \rbrace$ such that $k_i < k_{i+1}$ for every $i=1,\ldots, n$.
 By construction, $p_{\Theta}^{(s, l_s+1)}$ is the minor of a matrix in which all the variables $T_j$ for $j \in \Theta$ do not appear. In particular,  
 after permuting the rows and replacing the variables $T_r, \ldots, T_{n-1}$ by $T_{k_1}, \ldots, T_{k_{n-r}}$ keeping the same order, this matrix is equal to the matrix $B^{\star}$. This implies that, up to a sign, 
 $$ p_{\Theta}^{(s, l_s+1)}= \sum_{i=1}^{n+1} (-1)^{\beta_i} \left( \dfrac{T_{k_1} \cdots T_{k_{n+1}}}{T_{k_i}} \right) U_{\Theta \cup \lbrace k_i \rbrace}  = m_{\Theta},$$ 
 where each $\beta_i$ is determined by the permutations of the rows performed in the process, and equals $\alpha(\Theta,k_i)$. 
 
 This proves that each $m_{\Theta}$ is in  $I_{n}(B)$ for all sets $\Theta \subseteq \lbrace 1, \ldots, t \rbrace \setminus \lbrace n \rbrace$ with $|\Theta| = r-1$.
 Using now \cref{lemmaminors} iteratively as described in \cref{example1}, it follows that each $p_{\Psi}$ with $\,\Psi \subseteq \lbrace A_1, \ldots, A_{n-1} \rbrace \,$ is in the ideal generated by the minors of the form $m_{\Theta}$. Indeed, as in \cref{moreminors}, $p_{\Psi}= p_{\Theta}^{(1,1)}$ where $\Theta$ is the set of indexes corresponding to the columns in $\Psi$. Now, apply \cref{minorsformula} iteratively, starting from $p_{\Theta}^{(1,1)}$. The index in first term on the right side of \cref{minorsformula} increases at each iteration, until the term becomes a $p_{\Theta}^{(s, l_s+1)} = m_{\Theta}$. The second term on the right-hand side of \cref{minorsformula} is determined by a set of indexes obtained from one of those appearing in the previous iteration by replacing an index with a strictly smaller index. Hence, it eventually coincides with $m_{\Theta^{\star}}$.
 
 To conclude we only have to discuss the case $r \geq n$. Clearly all the columns $C_1, \ldots, C_r$ in the second part of the matrix are all equivalent up to permutation of the variables.
 Hence, similarly as in the previous case, the result on all the minors involving at least one of the first $n-1$ columns can be obtained by reducing to the case $r=n-1$. Finally, we only need to prove the statement for $n \times n$ minors involving only columns of the form $C_1, \ldots, C_r$. By renaming the variables, it is sufficient then to consider the matrix 
 $$ \footnotesize  \bmatrix 
     u_{11}T_{n+1} &   \ldots &  u_{n1}T_{2n} \\ 
     \vdots & \vdots &  \ldots  \\      
     u_{1,n-1}T_{n+1} &  \ldots &  u_{n,n-1}T_{2n} \\  
     (u_{1n}T_{n+1}-T_n) &  \ldots &  (u_{nn}T_{2n}-T_n)\\  
   \endbmatrix.  $$
 Expanding with respect to the last row, the determinant of this matrix is 
 $$  T_{n+1} \cdots T_{2n} \,U_{\Theta \cup \lbrace n \rbrace} + \sum_{i=1}^{n} (-1)^{n+i+1} \left( \dfrac{T_{n} \cdots T_{2n}}{T_{n+i}} \right) U_{\Theta \cup \lbrace n+ i \rbrace}   = m_{\Theta}$$ 
 where $\Theta= \lbrace 1, \ldots, n-1, 2n+1, \ldots, t \rbrace.$
\end{proof}

\begin{remark}
  \label{thetacontainingn} \hypertarget{thetacontainingn}{}
   Observe that also the polynomials $m_{\Theta}$ such that $n \in \Theta$ are in the ideal $I_n(B)$. Indeed any of such $m_{\Theta}$ can be expressed as linear combination with coefficients in $\lbrace 1, -1 \rbrace$ of generators of the form $m_{\Theta \setminus \lbrace n \rbrace \cup \lbrace k \rbrace}$, for $k \not \in \Theta$.
\end{remark}
  
\begin{remark}
  \label{vanishingofmtheta} \hypertarget{vanishingofmtheta}{}
  Similarly as in \cref{minorsU}, let
  $ \Theta = \lbrace i_1, \ldots, i_h, j_1, \ldots, j_{r-1-h} \rbrace$ be a set of indexes such that $i_1, \ldots, i_h \leq n$ and $j_1, \ldots, j_{r-1 -h} \geq n+1$. Call $M$ the submatrix of $U$ of size $h \times (h+1)$ obtained by taking rows $i_1, \ldots, i_h$ and removing the columns $j_1, \ldots, j_{r-1-h}$. Then, the polynomial $m_{\Theta}$ is zero if and only if the rank of $M$ is $< h$. 

Indeed, by \cref{minorsB}, $m_{\Theta}= 0$ if and only if for every $k_i \not \in \Theta$, the minor $U_{\Theta \cup \{ k_i\}}= 0$. This is equivalent to say that all the submatrices of $M$ of size $h \times h$ and all the submatrices of $U$ of size $(h+1) \times (h+1)$ containing $M$ are simultaneously singular. Hence this means that $\mbox{rank}(M) < h$.
  \end{remark}

\section{Linear star configurations of height two}
\label{sectionheight2}

In this section we exploit the results of \cref{sectionJacobiandual} to determine the defining ideal of the Rees algebra of ideals of linear star configurations of height two. In particular, in \cref{intersection} we relate the non-linear equations identified in \cite[3.5 and 4.2]{GSimisT} (see \cref{garrousian-simis-tohaneanu}) to the associated primes of the ideal of maximal minors of the Jacobian dual. 

\subsection{Defining ideal of the Rees algebra}

Our first goal is to identify an ideal $\mathcal{P}$, defined in terms of the polynomials $m_{\Theta}$, as the candidate for the non-linear part of the defining ideal of the Rees algebra of $I_{2,\mathcal{F}}$. The generators of this ideal $\mathcal{P}$ are introduced in the following lemma.

\begin{lem}
\label{irreduciblefactor} \hypertarget{irreduciblefactor}{}
Let $\Theta$ and $m_{\Theta}$ be defined as in \cref{minorsB}. Suppose $m_{\Theta} \neq 0$. Then $m_{\Theta} = f h_{\Theta}$ where $f$ is either a unit or a squarefree monomial in the variables $T_i$ and $h_{\Theta}$ is an irreducible nonzero and non-monomial element of $k[T_1, \ldots, T_t]$.
\end{lem}

\begin{proof}
Let $k_1, \ldots, k_{n+1}$ be the indexes not belonging to $\Theta$. Clearly the variable $T_{k_i}$ divides $m_{\Theta}$ if and only if $U_{\Theta \cup \lbrace k_i \rbrace} = 0$. For simplicity rename $U_i:= U_{\Theta \cup \lbrace k_i \rbrace}$. By relabeling the indexes, we can assume that there exists $e \geq 2$ such that $U_i \neq 0$ for $i \leq e$ and $U_i=0$ for $i > e$. 
Indeed by assumption $m_{\Theta} \neq 0 $ and, by \cref{minorsJacDual}, it is in the defining ideal of the Rees algebra of $I_{2, \mathcal{F}}$. Hence $m_{\Theta} $ cannot be a monomial in the variables $T_1, \ldots, T_t$ and therefore at least two minors $U_i$ are nonzero. Now, if $e=n+1$, then $h_{\Theta}=m_{\Theta}.$
Otherwise define 
\begin{equation} 
  \label{htheta} \hypertarget{htheta}{}
    h_{\Theta} \coloneq \frac{m_{\Theta}}{T_{k_{e+1}} \cdots T_{k_{n+1}}}. 
\end{equation} 
We have now that $h_{\Theta}$ can be expressed as $h_{\Theta}= \alpha T_{k_1} + \beta$ where $\alpha= \sum_{i=2}^e (T_{k_2} \cdots T_{k_e}) T_{k_i}^{-1} U_i$  and $\beta= T_{k_2} \cdots T_{k_e} U_1 $. Since $U_i \neq 0$ for every $i \leq e$, we get that $\alpha$ and $\beta$ have no common factors and $h_{\Theta}$ is irreducible.
\end{proof}

\begin{definition}
\label{idealP} \hypertarget{idealP}{}
For every $\Theta$ defined as in \cref{minorsB}, let $h_{\Theta}$ be defined as in \cref{irreduciblefactor}. We denote by $\mathcal{P}$ the ideal generated by the $h_{\Theta}$.
\end{definition}

Notice that the ideal $\mathcal{P}$ is contained in the non-linear part of the defining ideal $\mathcal{J}$ of the Rees algebra of $I_{2,\mathcal{F}}$. Indeed, since $I_n(B) \subseteq \mathcal{J}$, by \cref{minorsJacDual} and \cref{irreduciblefactor}, $m_{\Theta} = f h_{\Theta} \in \mathcal{J}$. But $f$ is either a unit or a squarefree monomial in the variables $T_1, \ldots, T_t$ and cannot be in $\mathcal{J}$. 
Since $\mathcal{J}$ is prime, it follows that $h_{\Theta} \in \mathcal{J}$.

Moreover, observe that if $I_{2,\mathcal{F}}$ satisfies the $G_n$ condition, then $\mathcal{P} = I_n(B)$ and this coincides with the 
non-linear part of $\mathcal{J}$ by \cref{morey-ulrich}. In \cref{orlikterao} we prove that in general $\mathcal{L} + \mathcal{P} = \mathcal{J}$, however $I_{2,\mathcal{F}}$ may no longer satisfy $G_n$. The following example shows that when the $G_n$ condition is not satisfied, one might have a proper containment $ I_n(B) \subsetneq \mathcal{P} $.

\begin{example}
  Let $R=K[x_1, x_2, x_3, x_4]$ and $\mathcal{F} = \lbrace x_1, x_2, x_3, x_4, L_1, L_2 \rbrace, $  where $L_1 = x_1+ x_2+x_3+x_4$ and $L_2= x_2+2x_3+3x_4$. Then, by \cref{Gngeneral} the ideal $I_{2, \mathcal{F}}$ does not satisfy the $G_n$ condition. Notice that 
  $$ U= \bmatrix 
   1 &   0 \\ 
   1 &   1  \\      
   1  &  2 \\  
   1  &  3 \\  
   \endbmatrix \quad \mathrm{and} \quad B= \bmatrix 
   T_1 &   0 &  0 & T_5 & 0 \\ 
   -T_2 & T_2 &  0 & T_5 & T_6  \\      
   0 &  -T_3 &  T_3 & T_5 & 2T_6 \\  
   0 &  0 &  -T_4 &  T_5 - T_4 & 3T_6-T_4 \\  
   \endbmatrix. $$
  In this case, 
  $$\mathcal{L}= (x_1T_1-x_2T_2, x_2T_2-x_3T_3, x_3T_3-x_4T_4, x_4T_4-L_1T_5, x_4T_4-L_2T_6) $$
  and $I_4(B)$ is generated by:
  \begin{eqnarray*}
    m_6 \!\! & = & T_1T_2T_3T_5 -  T_1T_2T_3T_4 + T_1T_2T_4T_5 + T_1T_4T_3T_5 + T_4T_2T_3T_5, \\
    m_5 \!\! & = & 3T_1T_2T_3T_6 -  T_1T_2T_3T_4 + 2T_1T_2T_4T_6 + T_1T_4T_3T_6, \\
    m_3 \!\! & = & T_1T_2T_4T_5 -  2T_1T_2T_4T_6 - T_1T_2T_5T_6 + T_1T_4T_5T_6 + 2T_2T_4T_5T_6,\\
    m_2 \!\! & = & T_1T_4T_3T_5 -  T_1T_6T_3T_4 - T_1T_6T_4T_5 - 2T_1T_6T_3T_5 + T_4T_6T_3T_5, \\
    m_1 \!\! & = & -T_4T_2T_3T_5  +2 T_6T_2T_4T_5 + T_6T_4T_3T_5 +3 T_6T_2T_3T_5.
  \end{eqnarray*}
  Also, $\,\displaystyle{h_1= \frac{m_1}{T_5}= \frac{m_5}{T_1} }\,$ and $\,\mathcal{P}=(m_6, m_3,m_2, h_1) \supsetneq I_4(B)$. \cref{orlikterao} will show that the defining ideal of $\mathcal{R}(I_{2, \mathcal{F}})$ is $\mathcal{L} + \mathcal{P}$.
  Each of the polynomials $h_{i}$ has the form $\partial D_i$ for dependency $D_i$ among the elements $x_1, x_2, x_3, x_4, L_1, L_2$ as in \cref{dependency} and \cref{deltadependency}. We have:
  \begin{eqnarray*}
     D_6 \!\!\! & \colon & x_1+x_2+x_3+x_4-L_1 = 0, \\
     D_5 = D_1 \!\!\! & \colon & x_2+2x_3+3x_4-L_2 = 0, \\
    D_3 \!\!\! & \colon & 2x_1+x_2-x_4-2L_1+L_2 = 0, \\
     D_2 \!\!\! & \colon & x_1-x_3-2x_4-L_1+L_2=0.
  \end{eqnarray*}
\end{example}

In the previous example, the non-linear part of the defining ideal of $\mathcal{R}(I_{2, \mathcal{F}})$ is generated by the polynomials $\partial D$ corresponding to the minimal dependencies $D$ among the elements of $\mathcal{F}$. This is true in general. Indeed, in \cref{orlikterao} below, we show that the defining ideal of $\mathcal{R}(I_{2,\mathcal{F}})$ is $\mathcal{L} + \mathcal{P}$, where $\mathcal{L}=(\lambda_1, \ldots, \lambda_{t-1})$ is the ideal of linear relations of $I_{2,\mathcal{F}}$. From the presentation matrix of $I_{2,\mathcal{F}}$ it is clear that $ \displaystyle{\lambda_i= F_i T_i - F_{i+1} T_{i+1}} $ for every $i=1, \ldots, t-1$.

\begin{lem}
 \label{zerorowlemma}\hypertarget{zerorowlemma}{}
  With the assumptions and notations of \cref{linearsetting}, let $I_{2,\mathcal{F}}$ be a linear star configuration of height 2. Assume that, up to reordering the variables, the first row of the matrix $U$ is zero.  
  Set $\mathcal{G}= \lbrace x_2, \ldots, x_n, L_1, \ldots, L_r  \rbrace.$  Let $B$ and $B^*$ be Jacobian dual matrices for $I_{2, \mathcal{F}}$ and $I_{2,\mathcal{G}}$ respectively. Then $$ I_n(B)=(T_1) I_{n-1}(B^*). $$ Moreover, the ideal $\mathcal{P}$ defined in \cref{idealP} for $I_{2, \mathcal{F}}$ and for $I_{2,\mathcal{G}}$ is the same.
\end{lem}

\begin{proof}
 Expressing $B$ and $B^*$ as in \cref{jacobiandual}, it is easy to observe that 
  $$ \footnotesize B= \bmatrix T_1 & 0 & \ldots & 0 \\ -T_2 & & &  \\ 0 & & B^* & \\ \vdots & &  & \\
     0 & & & 
  \endbmatrix. $$ 
 Both statements now follow from the definitions.
\end{proof}

\begin{thm}
\label{orlikterao}\hypertarget{orlikterao}{}
With the assumptions and notations of \cref{linearsetting}, let $I_{2,\mathcal{F}}$ be a the ideal of a linear star configuration of height two. Then, the ideal $\mathcal{P}$ is the non-linear part of the defining ideal of the Rees algebra of $I_{2,\mathcal{F}}$. In particular $\mathcal{J}= \mathcal{L} + \mathcal{P}.$
\end{thm}

\begin{proof}
By \cref{garrousian-simis-tohaneanu} it is sufficient to prove that the polynomial $\partial D$ associated to any dependency $D$ among the elements of $\mathcal{F}$ is in the ideal $\mathcal{P}$. 

First we show that every polynomial $h_{\Theta}$ 
is of the form $\partial D$ for some dependency $D$.
Indeed consider the natural map $\varphi \colon R[T_1, \ldots, T_t] \to \mathcal{R}(I_{2, \mathcal{F}})$  and let $G= \prod_{i=1}^t F_i$. Then, using \cref{htheta}, we have that
  $$0= \varphi(h_{\Theta}) = \sum_{i=1}^e (-1)^{\alpha(\Theta, i)}U_{\Theta \cup \lbrace k_i \rbrace} \varphi\Big(\frac{T_{k_1}\cdots T_{k_e}}{T_{k_i}}\Big)  =  \sum_{i=1}^e (-1)^{\alpha(\Theta, i)}U_{\Theta \cup \lbrace k_i \rbrace} G^{\,e-2} F_{k_i}. $$
It follows that, after dividing by $G^{e-2}$, 
the last term is a dependency $D$ among the elements of $\mathcal{F}$ in the sense of \cref{dependency}. 
Therefore, $h_{\Theta}$ is the corresponding polynomial $\partial D$ as defined in \cref{deltadependency}.

To prove that, for any dependency $D$, $\partial D$ is in $\mathcal{P}$ we work by induction on $r \geq 1$.
If $r=1$, up to multiplying by units, there is only one dependency $D$ and clearly $\partial D = u h_{\Theta}$ with $\Theta = \emptyset$ and for some $u \in K$.

Assume now that $r \geq 2$ and notice that any dependency $D$ can be written as
$$ D \colon a_1 L_1 + \ldots + a_r L_r + b_1 x_1 + \ldots + b_n x_n = 0,  $$ where the coefficients $b_j$ are uniquely determined after $a_1, \ldots, a_r \in K$ are chosen. Using the inductive hypothesis we can deal with all the dependencies such that at least one of the coefficients $ a_1, \ldots, a_r $ is zero. For simplicity assume that $a_r=0$ and consider the star configuration $I_{2,\mathcal{F'}}$ where $\mathcal{F'}:= \mathcal{F} \setminus \lbrace L_r \rbrace$. 
Let $\mathcal{P'}$ be the ideal generated by the polynomials $h_{\Theta}$ constructed for $I_{2,\mathcal{F'}}$ as in \cref{idealP}. 
We can look at $\mathcal{P'}$ as an ideal of $K[T_1, \ldots, T_t]$.
Using \cref{jacobheight2}, \cref{minorsJacDual} and \cref{irreduciblefactor}, it can be easily checked that $\mathcal{P'} \subseteq \mathcal{P}$. Now, all the dependencies $D$ such that $a_r=0$ are also dependencies among the elements of $\mathcal{F'}$. Hence, by the inductive hypothesis, the corresponding polynomials $\partial D \in \mathcal{P'} \subseteq \mathcal{P} $. 


To conclude we can restrict to the case where $ a_1, \ldots, a_r \neq 0 $, and since the polynomial $\partial D$ is unique up to multiplying scalars, we can further assume that $a_1=1$. Let now $U$ be the $n \times r$ matrix of the coefficients $u_{ij}$ as in \cref{linearsetting}.
By \cref{zerorowlemma} we can always reduce to the case in which no rows of $U$ are zero. Hence, once fixed such $ a_1, \ldots, a_r $, 
let $\chi \subseteq \lbrace 1, \ldots, n \rbrace$ be a (possibly empty) maximal set of indexes such that $b_j=0$ for every $j \in \chi$ and 
the rows of the matrix $U$ indexed by the elements of $\chi$ are linearly independent. 
By definition $|\chi| \leq n$.
We show also that $|\chi| \leq r-1$. Indeed, by way of contradiction and by relabeling, say that $\chi \supseteq \lbrace  1, \ldots, r \rbrace$. Thus $b_1, \ldots, b_r=0$, which implies that the linear forms $x_{r+1}, \ldots, x_n, L_1, \ldots, L_r$ are not linearly independent. But by  \cref{minors-regseq} the assumption that the first $r$ rows of $U$ are linearly independent implies that $x_{r+1}, \ldots, x_n, L_1, \ldots, L_r$ form a regular sequence, a contradiction.

Without loss of generality, say now that $\chi = \lbrace  1, \ldots, h \rbrace$ with $0 \leq h \leq \min \lbrace r-1, n \rbrace$. The dependency $D$ becomes $ \, D \colon L_1 + a_2 L_2+ \ldots + a_r L_r + b_{h+1} x_{h+1} + \ldots + b_n x_n = 0$. If $h= r-1$, the equations with respect to $x_1, \ldots, x_{r-1}$ determine a linear system in $r-1$ equations $\, -u_{1k} = a_2 u_{2k} + \ldots + a_r u_{rk}\, $ for $k= 1,\ldots, r-1$ and $r-1$ indeterminates $ a_2, \ldots, a_{r}$. The assumption that the first $r-1$ rows of $U$ are linearly independent forces this system to have a unique solution. Hence, up to multiplying units, $D$ is the only dependency related to such set $\chi$, and necessarily, setting $\Theta = \chi$, the corresponding polynomial $\partial D$ is $h_{\Theta}$.

Suppose now that $h < r-1$. Since the first $h$ rows of $U$ are linearly independent, by permuting the columns we may assume that the minor $W:= U_{\chi \cup \lbrace n+1, \ldots, n+r-h \rbrace} $ is nonzero. 
For $i=1, \ldots, r-h$ define $$ \Theta_i := \chi \cup \lbrace n+1, \ldots, n+r-h \rbrace \setminus \lbrace n+i \rbrace. $$ By construction $|\Theta_i| = r-1$ and $h_{\Theta_{i}}$ is well-defined. We claim that 
$$ \partial D = \sum_{i=1}^{r-h} \frac{a_i}{W} \left( \frac{T_{n+1} \cdots T_{n+r-h}}{T_{n+i}} \right) h_{\Theta_{i}} \in \mathcal{P}. 
  $$ 
  To do this we need to check that the coefficients of each term coincide. The quantity on the right-hand side is a sum of terms of the form $ c_j (T_{h+1} \cdots T_{n+r})(T_j^{-1}) $ for $j \geq h+1$. We have to prove that $c_j =  b_j $ if $j \leq n$ and $c_j= a_{j-n}$ if $j \geq n+1$. 
  We consider different subcases. \\
  \textbf{Case (i)}: $n+1 \leq j \leq n+r-h$. This term appears only once among the terms of $h_{\Theta_{j-n}}$. Observe that by \cref{minorsB}, $\,\alpha(\Theta_i, n+i)= -h$.
  Thus the coefficient of the term we are considering is $\,(-1)^{h} U_{\Theta_{j-n} \cup \lbrace j \rbrace} = (-1)^{h}W$. Hence, $\,c_j= (-1)^{h}a_{j-n}(W^{-1})W = (-1)^{h}a_{j-n}$. \\
  \textbf{Case (ii)}: $j > n+r-h $. For $k= 1,\ldots, h$, consider the linear system on the $h$ equations 
  $$a_{r-h+1}u_{r-h+1,k} + \ldots + a_r u_{rk} = -(a_1u_{1,k} + \ldots + a_{r-h}u_{r-h,k}).$$ 
  Set $\sigma_{i,j} \coloneq \alpha(\Theta_i,n+i)+ \alpha(\Theta_i,j)$. Observe that by \cref{minorsB} $\,\sigma_{i,j}= j - (n+r-h)$.
  By Cramer's rule we get $$ c_j = (-1)^h \sum_{i=1}^{r-h} (-1)^{\sigma_{i,j}} \frac{a_i}{W} \, U_{\Theta_i \cup \lbrace j \rbrace}= (-1)^h a_{j-n}. $$
  \textbf{Case (iii)}: $h < j \leq n $.  Notice that $b_j = -(a_1u_{1j}+ \ldots + a_ru_{rj}) $ and $$c_j = \sum_{i=1}^{r-h} (-1)^{\sigma_{i,j}} \frac{a_i}{W} \, U_{\Theta_i \cup \lbrace j \rbrace},$$ where in this case $\sigma_{i,j}=-h+1$.
  Computing the minor $U_{\Theta_i \cup \lbrace j \rbrace}$ with respect to the $j$-th row, we express $$  U_{\Theta_i \cup \lbrace j \rbrace} = u_{ij} W + \sum_{k=r-h+1}^r (-1)^{r-h+k} u_{kj} \,U_{\Theta_i \cup \lbrace n+k \rbrace}. $$ Hence, by replacing $U_{\Theta_i \cup \lbrace j \rbrace}$ in the equation for $c_j$ and applying (ii) we get
    \begin{eqnarray*}
      c_j  & = & (-1)^{1-h} \Big(\sum_{i=1}^{r-h}  a_i u_{ij} + \sum_{i=1}^{r-h}  \frac{a_i}{W} \sum_{k=r-h+1}^r (-1)^{r-h+k} u_{kj} U_{\Theta_i \cup \lbrace n+k \rbrace}\Big) \\
       & = & (-1)^{1-h} \Big(\sum_{i=1}^{r-h} a_i u_{ij} + \sum_{k=r-h+1}^r  u_{kj} a_k \Big) = (-1)^{h} b_j.
    \end{eqnarray*}
  This concludes the proof after multiplying all the $c_j$ by $(-1)^h$.
\end{proof}

 \begin{remark}
    \label{degreesOT} \hypertarget{degreesOT}{}
  From \cref{orlikterao} and its proof it follows that the polynomials $h_{\Theta}$ defined in \cref{irreduciblefactor} are a minimal generating set for the non-linear part $\mathcal{P}$ of the defining ideal $\mathcal{J}$ of $\mathcal{R}(I_{2, \mathcal{F}})$. Moreover, \cref{htheta} provides an explicit formula for each $h_{\Theta}$. In particular, the degrees of the non-linear equations of the Rees algebra can be explicitly calculated from the coefficients of the linear forms $\lbrace L_1, \ldots, L_r \rbrace$ and are always between 4 and $n$. In fact, since any three of $x_1, \ldots, x_n, L_1, \ldots, L_r$ are a regular sequence, there is no dependency of degree at most 3. 
 \end{remark}

  \begin{cor}
    \label{corr=1} \hypertarget{corr=1}{}
    Assume $r=1$. Set $\displaystyle{\mathcal{F}= \lbrace x_1, \ldots, x_n, u_{e+1}x_{e+1}+\ldots+ u_nx_n \rbrace}$ 
with $e \geq 0$ and $u_i \neq 0$ for all $e+1 \leq i \leq n$. Then, the defining ideal of $\mathcal{R}(I_{2,\mathcal{F}})$ is equal to $\mathcal{L}+ (f)$ where 
$$ f= T_{e+1}\cdots T_{n} - \sum_{i=e+1}^n \left( \frac{T_{e+1}\cdots T_{n+1}}{T_i} \right) u_i.   $$ 
  \end{cor}

\subsection{Primary decomposition}

The aim of this subsection is to interpret the ideal $\mathcal{P}$ in terms of the primary decomposition of $I_n(B)$. We have already observed that when $I_{2, \mathcal{F}}$ satisfies the $G_n$ condition, $I_n(B) = \mathcal{P}$ is the defining ideal of the fiber cone of $I_{2, \mathcal{F}}$, thus a prime ideal.

When the $G_n$ condition is no longer satisfied, $I_n(B)$ is not a prime ideal and we believe that $\mathcal{P}$ is one of its minimal primes. In particular we state the following conjecture.
\begin{conj}
The ideal  $\mathcal{P}$ is the only associated prime of $I_n(B)$ not generated by monomials.
\end{conj}
We prove that the conjecture is true under an additional assumption on the matrix of coefficients $U$.
We start by some observing some properties of the height of $I_n(B)$.


\begin{remark}
  \label{eagon-northcott} \hypertarget{eagon-northcott}{}
  By the Eagon-Northcott Theorem \cite[Theorem 1]{EN}, the ideal $I_n(B)$ has height $\leq r$.
  This upper bound is met if $I_{2,\mathcal{F}}$ satisfies the $G_n$ condition. Indeed, in this case, by \cref{morey-ulrich}, $I_n(B)$ is the defining ideal of the fiber cone of $I_{2,\mathcal{F}}$. Set $\mathcal{F}'= \mathcal{F} \setminus \{ F_t \}$, and call $B'$ the Jacobian dual matrix of $I_{2,\mathcal{F}'}$. By \cref{Gngeneral}, it is easy to observe that also $I_{2,\mathcal{F}'}$ satisfies condition $G_n$ and by \cref{minorsJacDual}, $I_n(B')S \subsetneq I_n(B)$. The fact that these ideals are primes and an inductive argument on $r$ imply that $\het I_n(B) = r$. 
  \end{remark}

However, next result shows that if we remove the $G_n$ assumption, then the height of $I_{n}(B)$ can be arbitrarily smaller than $r$. In the case when $I_{2,\mathcal{F}}$ does not satisfy the $G_n$ condition, by \cref{sregseq-minorsU} the matrix of coefficients $U$ must have some zero minor. Next lemma shows that the presence of such zero minors corresponds to containments of $I_n(B)$ in some monomial prime ideals of height at most $r$.


 \begin{lem}
 \label{monomialprimes} \hypertarget{monomialprimes}{}
Consider a set of indexes $\chi \subseteq \lbrace 1, \ldots, t \rbrace$ such that $|\chi| \leq r$. Let $\mathfrak{p}_{\chi}$ be the prime ideal of $k[T_1, \ldots, T_t]$  generated by the variables $T_k$ for $k \in \chi$.
The following are equivalent:
\begin{enumerate}[\rm(1)]
\item The ideal $I_{n}(B) \subseteq \mathfrak{p}_{\chi}$.
\item Each minor of $U$ of the form $U_{\Omega}$ with $\chi \subseteq \Omega$ is zero.
\end{enumerate}
In particular, in the case $|\chi| = r$ this gives that $I_{n}(B) \subseteq \mathfrak{p}_{\chi}$ if and only if $U_{\chi}=0$.
 \end{lem}

\begin{proof}
By \cref{minorsJacDual}, the ideal $I_n(B)$ is generated by the polynomials $m_{\Theta}$ for all the sets of indexes $\Theta$ such that $|\Theta| = r-1, n \not \in \Theta$. If there exists at least two indexes $k_i, k_j \in \chi \setminus \Theta$, then each monomial of $m_{\Theta}$ is divisible either by $T_{k_i}$ or by $T_{k_j}$ or by both and therefore $m_{\Theta} \in \mathfrak{p}_{\chi}$.

Hence, we need to discuss only the case $\chi \subseteq \Theta \cup \lbrace k_i \rbrace$ for some $k_i \not \in \Theta$. If $k_i \in \chi$, all the monomials of $m_{\Theta}$ except one are divisible by $T_{k_i}$ but none of them is divisible by any other variable $T_{k_j}$ generating $\mathfrak{p}_{\chi}$. The only monomial of $m_{\Theta}$ not divisible by $T_{k_i}$ has coefficient $U_{\Theta \cup \lbrace k_i \rbrace}$. Hence $m_{\Theta} \in \mathfrak{p}_{\chi}$ if and only if $U_{\Theta \cup \lbrace k_i \rbrace}=0$.

If instead we assume $\chi \subseteq \Theta$, we have that none of the monomials of $m_{\Theta}$ is in $\mathfrak{p}_{\chi}$ and the only possibility to have $m_{\Theta} \in \mathfrak{p}_{\chi}$ is to have $U_{\Theta \cup \lbrace j \rbrace}=0$ for every $j \not \in \Theta$. The thesis follows since this must hold for every $\Theta$.
\end{proof}

 To describe the primary decomposition of $I_n(B)$, in order to avoid too much technicality, we focus on the case when a suitable condition on the matrix $U$ (but weaker than the $G_n$ condition) is satisfied.

\begin{remark}
\label{assumption} \hypertarget{assumption}{}
    By \cref{vanishingofmtheta} and \cref{monomialprimes} the following are equivalent:
\begin{enumerate}[\rm(1)]
\item $m_{\Theta} \neq 0$ for every $\Theta$.
\item For every $h \geq 1$, every submatrix of size $h \times (h+1)$ of $U$ has maximal rank.
\item $I_n(B)$ is not contained in any monomial prime ideal of height $< r$.
\end{enumerate}
Notice that if $r=1$, these equivalent conditions are always satisfied, while if $r=2$ this are satisfied if and only if no row of $U$ is zero. \cref{zerorowlemma} shows that, to study the primary decomposition of $I_n(B)$, we can always reduce to assume that the matrix $U$ has no zero rows. Thus if $r \leq 2$ our proof covers all the possible cases.
\end{remark}

We now need a couple of lemmas to show that, for some distinct sets $\Theta$, the corresponding $h_{\Theta}$ are associated in the polynomial ring $K[T_1,\ldots, T_t]$. First we explore further the relation between the $h_{\Theta}$'s and dependencies pointed out in \cref{orlikterao}.

\begin{lem}
 \label{hthetadependency} \hypertarget{hthetadependency}{}
 Let $\Theta$, $m_{\Theta}$, $h_{\Theta}$ be defined as in \cref{minorsB} and \cref{irreduciblefactor}. If $h_{\Theta} \neq 0$, then, up to multiplying by a factor in $K$, there exists a unique dependency $D: a_1 L_1 + \ldots + a_r L_r + b_1 x_1 + \ldots + b_n x_n = 0$ such that $b_j=0$ for $j \in \Theta$, $j \leq n$ and $a_{n-j}=0$ for $j \in \Theta$, $j > n$. In particular $h_{\Theta} = \partial D$.
\end{lem}

\begin{proof}
 The case $r=1$ is clear. Applying the same inductive argument on $r$ as in the proof of \cref{orlikterao}, we reduce to prove the statement in the case $\Theta = \lbrace j_1, \ldots, j_{r-1} \rbrace \subseteq \lbrace 1, \ldots, n \rbrace$. Since $h_{\Theta} \neq 0$, also $m_{\Theta} \neq 0$ and by \cref{vanishingofmtheta} this implies that the rows $j_1, \ldots, j_{r-1}$ of the matrix $U$ are linearly independent. Again, as in the proof of \cref{orlikterao}, this condition implies that, up to multiplying a scalar, there exists a unique dependency $D$ such that $b_{j_1}, \ldots, b_{j_{r-1}}=0$. 
Necessarily $h_{\Theta} = \partial D$.
\end{proof}

\begin{lem}
\label{associated} \hypertarget{associated}{}
Let $\Theta$, $m_{\Theta}$, $h_{\Theta}$ be defined as in \cref{minorsB} and \cref{irreduciblefactor}. Suppose that $ h_{\Theta} \neq m_{\Theta} \neq 0 $. 
Given $j \in \Theta$ and $k \not \in \Theta$ such that the minor $U_{\Theta \cup \lbrace k \rbrace} = 0$, define $\Theta':= \Theta \setminus \lbrace j \rbrace \cup \lbrace k \rbrace$.
Then, if $m_{\Theta'} \neq 0$, there exists a unit $a \in K$ such that $h_{\Theta} = a h_{\Theta'}.$ 
\end{lem}
\begin{proof}
Let $\Theta = \lbrace j_1, \ldots, j_{r-1} \rbrace$ and
let $\lbrace k_1, \ldots, k_{n+1} \rbrace = \lbrace 1, \ldots, t \rbrace \setminus \Theta$. Since $0 \neq m_{\Theta} \neq h_{\Theta} $ by reordering the indexes, there exists $1 \leq e < n $ such that $U_{\Theta \cup \lbrace k_l \rbrace} = 0$ for $l \leq e$ and $U_{\Theta \cup \lbrace k_l \rbrace} \neq 0$ for $l > e$. 
Take $k \in \lbrace k_1, \ldots, k_{e} \rbrace$ and $j \in \lbrace j_1, \ldots, j_{r-1} \rbrace$.
Observe that $ U_{\Theta' \cup \lbrace j \rbrace} = U_{\Theta \cup \lbrace k \rbrace} = 0$.
By definition of $m_{\Theta}$ we can write $$ \frac{m_{\Theta}}{T_{k}} = \sum_{l =1}^{n+1} (-1)^{\alpha(\Theta,l)}\left( \frac{T_{k_1}\cdots T_{k_{n+1}}}{T_{k_l} T_{k}} \right) U_{\Theta \cup \lbrace k_l \rbrace}, $$ $$ \frac{m_{\Theta'}}{T_{j}} = \sum_{l =1}^{n+1} (-1)^{\alpha(\Theta',l)}\left( \frac{T_{k_1}\cdots T_{k_{n+1}}}{T_{k_l} T_{k}} \right) U_{\Theta' \cup \lbrace k_l \rbrace}.  $$
Hence, to conclude we have to show that there exists a unit $a \in K$ such that for every $l \in \lbrace k_1,\ldots, k_{n+1} \rbrace \setminus \lbrace k \rbrace$,
 \begin{equation}
   \label{signformula} \hypertarget{signformula}{}
      U_{\Theta' \cup \lbrace k_l \rbrace}  = (-1)^{\alpha(\Theta,k_l)-\alpha(\Theta',k_l)} a U_{\Theta \cup \lbrace k_l \rbrace}.
 \end{equation}

This equality will also imply that one of these two minors is zero if and only if so is the other one and the thesis will then follow by \cref{irreduciblefactor}. 

Using \cref{hthetadependency}, it is sufficient to show that $h_{\Theta}$ and $h_{\Theta'}$ correspond to the same dependency $D$. Let $D: a_1 L_1 + \ldots + a_r L_r + b_1 x_1 + \ldots + b_n x_n = 0$ be the dependency such that $h_{\Theta}= \partial D$.
For simplicity set $c_l\coloneq b_l$ for $l \leq n$ and $c_l \coloneq a_{n-l}$ for $l > n$.
Hence $c_{j_l}= 0$ for every $l=1, \ldots, r-1$ and moreover, since $U_{\Theta \cup \lbrace k \rbrace} = 0$ then also $c_k=0$. Since also $m_{\Theta'} \neq 0$, then $h_{\Theta'} \neq 0$ and by \cref{hthetadependency}, the dependency $D'$ associated to $h_{\Theta'}$ is then equal to $aD$ for some nonzero $a \in K$. It follows that $h_{\Theta'}=ah_{\Theta}$.
\end{proof}

Next theorem shows that, under the assumption that all the $m_{\Theta}$ are nonzero (which by \cref{assumption} corresponds to a maximal rank condition on certain submatrices of $U$), the ideal $I_n(B)$ can be obtained intersecting $\mathcal{P}$ with all the monomial primes of height at most $r$ containing $I_n(B)$. Following the notation of \cref{monomialprimes}, define $\Lambda $ to be the set of all the monomial prime ideals of $K[T_1, \ldots, T_t]$ of the form $\mathfrak{p}_{\chi}$ for $\chi \subseteq \lbrace 1, \ldots, t \rbrace$, $|\chi| \leq r$ and such that $I_{n}(B) \subseteq \mathfrak{p}_{\chi}$.

  
\begin{thm} \label{intersection} \hypertarget{intersection}{}
Consider the set $\Lambda$ defined above and call $Q= \bigcap_{\mathfrak{p} \in \Lambda} \mathfrak{p}.$ Let $\mathcal{P} $ be defined as in \cref{idealP}. Suppose that $m_{\Theta} \neq 0$ for every $\Theta$.
Then  $I_n(B)$ is radical and its primary decomposition is \begin{equation}
    \label{primary} \hypertarget{primary}{}
     I_n(B)= Q \cap \mathcal{P}.
\end{equation}
\end{thm}
  
  \begin{proof}
 By \cref{garrousian-simis-tohaneanu} and \cref{orlikterao}, the ideal $\mathcal{P}$ is the defining ideal of the fiber cone of $I_{2, \mathcal{F}}$, hence it is prime of height $r$.
The inclusion $I_n(B) \subseteq Q \cap \mathcal{P}$ follows by \cref{monomialprimes} and \cref{irreduciblefactor}.
For the other inclusion consider an element $\alpha = \sum \alpha_i h_{\Theta_i} \in Q \cap \mathcal{P}$. Since each $m_{\Theta} \in I_n(B)$, we can restrict to consider only the case in which $h_{\Theta_i} \neq m_{\Theta_i}$ for every $i$. By \cref{irreduciblefactor}, this implies that some minor of $U$ of the form $U_{\Theta_i \cup \lbrace k \rbrace} = 0$. Hence, setting $\chi= \Theta_i \cup \lbrace k \rbrace$, by \cref{monomialprimes} $\mathfrak{p}_{\chi} \in \Lambda$ and therefore $Q \subseteq \mathfrak{p}_{\chi}$. 
But now, none of the monomials of $h_{\Theta_i}$ is in $\mathfrak{p}_{\chi}$. Since $\alpha \in Q$ and $\mathfrak{p}_{\chi}$ is generated by variables, this clearly implies $\alpha_i \in \mathfrak{p}_{\chi}$. In particular, for every $i$, we get $$\alpha_i \in Q_{\Theta_i}:=\bigcap_{\scriptstyle \mathfrak{p}_{\chi} \in \Lambda_i} \mathfrak{p}_{\chi},$$ where $\Lambda_i$ is the set of primes in $\Lambda$ not containing $h_{\Theta_i}.$
Thus we can reduce to fix one set of indexes $\Theta$ such that $m_{\Theta} \neq  h_{\Theta} $, and show that $Q_{\Theta}h_{\Theta} \subseteq I_n(B)$. This would imply that $\alpha \in I_n(B)$ and the proof would be complete. 


 Call $k_1, \ldots, k_{n+1}$ the indexes not belonging to $\Theta$ and, by reordering, consider the positive integer $e$ such that $1 \leq e < n $ and $\,U_{\Theta \cup \lbrace k_l \rbrace} = 0$ if and only if $l \leq e$. Suppose $\Theta= \lbrace j_1, \ldots, j_{r-1} \rbrace$.
 Consider the sets $E= \Theta \cup \lbrace k_{1}, \ldots, k_{e} \rbrace $ and $\mathcal{E}= \lbrace T_i \mbox{ : } i \in E \rbrace$. Let $H:= I_{r,\mathcal{E}} \in K[T_1, \ldots, T_t]$ be the ideal of the star configuration of height $r$ on the set $\mathcal{E}$.
 
 We first prove that any associated prime $\mathfrak{p}$ of $H$ is in the set $\Lambda$ and $h_{\Theta} \not \in \mathfrak{p}$.
 This will imply that $Q_{\Theta} \subseteq H$ and we later prove that $H h_{\Theta} \subseteq I_n(B)$.
 
 By \cref{monomialprimes}, our first statement is true if we show that for every subset $\chi \subseteq E$ such that $|\chi|= r$, the minor $U_{\chi}= 0$ and $h_{\Theta} \not \in \mathfrak{p}_{\chi}$. This is true by assumption for the subsets $\chi$ containing $\Theta$.
 For the others sets, this can be done applying \cref{associated}, 
 inductively on the cardinality of $\chi \cap \lbrace k_{1}, \ldots, k_{e} \rbrace$. The basis of the induction is the case when $\Theta \subseteq \chi$.
 In the case $|\chi \cap \lbrace k_{1}, \ldots, k_{e} \rbrace|=2$, we get $\chi= \Theta' \cup \lbrace k_l \rbrace $ for some $\Theta' = \Theta \setminus \lbrace j \rbrace \cup \lbrace k_i \rbrace$ of the form considered in \cref{associated}. 
 Since $m_{\Theta'} \neq 0$, as a consequence of \cref{associated}, 
 we get that $h_{\Theta}$ and $h_{\Theta'}$ are associated the polynomial ring $K[T_1, \ldots, T_t]$ and therefore $U_{\chi}= 0$, since also $U_{\Theta \cup \lbrace k_l \rbrace}=0$. Moreover, by definitions $h_{\Theta'} \not \in \mathfrak{p}_{\chi}$, and thus also $h_{\Theta} \not \in \mathfrak{p}_{\chi}$.
 An iterative application of  \cref{associated}, 
using the fact that all the $m_{\Theta}$'s are nonzero, allows to deal in the same way with all the other cases. 

We now prove that $H h_{\Theta} \subseteq I_n(B)$. The minimal generators of $H$ have degree $e$ and are of the form $\beta_{\Omega}= \prod_{i \in E \setminus \Omega} T_i$ where  $\Omega \subseteq E$ and $|\Omega|= r-1$. For every of such sets $\Omega$, since $m_{\Omega} \neq 0$, as consequence of \cref{irreduciblefactor} and of \cref{associated},  we get $m_{\Omega}= \beta_{\Omega} h_{\Omega}$. Moreover, $h_{\Theta}$ and $h_{\Omega}$ are associated in $K[T_1, \ldots, T_t]$. Hence, for some unit $u \in K$, we have $ \beta_{\Omega} h_{\Theta}= \beta_{\Omega} uh_{\Omega} = u m_{\Omega} \in I_n(B) $.This implies $H h_{\Theta} \subseteq I_n(B)$ and concludes the proof. 
\end{proof}

  \section*{Acknowledgements}
  
  We thank Paolo Mantero for interesting conversations on ideals of star configurations that partly motivated this project, and Kuei-Nuan Lin for referring us to the work of Garrousian, Simis and Tohaneanu \cite{GSimisT}. We also are grateful to Alexandra Seceleanu for helpful feedback on a preliminary version of this preprint. 
  The third author is supported by the NAWA Foundation grant Powroty "Applications of Lie algebras to Commutative Algebra".

\end{document}